\RequirePackage{etoolbox}
\csdef{input@path}{{style/}{graphics/}}
\documentclass[MSNbibl,number,citesort,seceqn,dvips]{arxbj}

\aid{0}
\volume{21}
\issue{1}
\pubyear{2015}
\firstpage{505}
\lastpage{536}
\doi{10.3150/13-BEJ577} % kopijuoti is 'New paper accepted'

\makeatletter
\renewcommand{\emptyset}{\varnothing}
\newcommand{\rrVert}{\Vert}
\newcommand{\llVert}{\Vert}
\newcommand{\eqref}[1]{(\ref{#1})}
\newproclaim{hypo}{Assumption}[section]
\newtheorem{theo}[hypo]{Theorem}
\newproclaim{defi}[hypo]{Definition}
\newtheorem{lem}[hypo]{Lemma}
\newremark{Rq}[hypo]{Remark}

\newcommand{\TV}{\mathrm{TV}}
\newcommand{\eref}{\eqref}
\newcommand{\one}{\mathbf{1}}
\newcommand{\N}{\mathbb{N}}
\newcommand{\R}{\mathbb{R}}
\newcommand{\p}{\mathbb{P}}
\newcommand{\E}{\mathbb{E}}
\makeatother

\begin{document}
\begin{frontmatter}

\title{Exponential ergodicity for Markov processes with random switching}
\runtitle{Exponential ergodicity for Markov processes with random switching}

\begin{aug}
%%%% inicialai - be tarpu
\author[1]{\inits{B.}\fnms{Bertrand} \snm{Cloez}\corref{}\thanksref{1}\ead[label=e1]{bertrand.cloez@univ-mlv.fr}} \and
\author[2]{\inits{M.}\fnms{Martin} \snm{Hairer}\thanksref{2}\ead[label=e2]{M.Hairer@Warwick.ac.uk}}
%%\runauthor{} %% auto
\address[1]{Institut de Math\'{e}matiques de Toulouse,
Universit\'{e} de Toulouse,
F-31062 Toulouse, France.\\
\printead{e1}}
\address[2]{Mathematics Department, University of Warwick, Coventry, CV4 7AL, U.K.\\
\printead{e2}}
\end{aug}

% HISTORY:
\received{\smonth{3} \syear{2013}}
\revised{\smonth{9} \syear{2013}}

% ABSTRACT
%
\begin{abstract}
We study a Markov process with two components: the first component
evolves according to one
of finitely many underlying Markovian dynamics, with a choice of
dynamics that changes
at the jump times of the second component. The second component is
discrete and its jump rates
may depend on the position of the whole process.
%This process can model for example a spatial motion in an
%inhomogeneous random environment.
Under regularity assumptions on the jump rates and Wasserstein
contraction conditions for the underlying dynamics, we provide a
concrete criterion for the convergence to equilibrium in terms of
Wasserstein distance. The proof is based on a coupling argument and a
weak form of the Harris theorem. In particular, we obtain exponential
ergodicity in situations which do not verify any hypoellipticity assumption,
but are not uniformly contracting either. We also obtain a bound in
total variation distance under a suitable regularising assumption. Some
examples are given to illustrate our result, including a class of
piecewise deterministic Markov processes.
\end{abstract}

% KEYWORDS
% visi is mazosios raides ir pagal abecele
%

\end{frontmatter}

%s1 #&#
\section{Introduction}

Markov processes with switching are intensively used for modelling purposes
in applied subjects like biology \cite{CMMS,CDMR,FGM}, storage
modelling \cite{BKKD},
neuronal activity \cite{PMW,GM12}. This class of Markov processes is
reminiscent of
the so-called iterated random
functions \cite{DF99} or branching processes in random environment
\cite{S68} in the discrete time setting. Several recent works \cite
{BH12,BGM,MLMZ-Horm,MLMZ-Quant,CD08,SY05,GIY,GG96} deal with their
long time behaviour (existence of an invariant probability measure,
Harris recurrence, exponential ergodicity, hypoellipticity\ldots). In
particular, in \cite{BH12,MLMZ-Horm}, the authors provide a kind of
hypoellipticity criterion with H\"{o}rmander-like bracket conditions.
Under these conditions, they deduce the uniqueness and absolute
continuity of the invariant measure, provided that a suitable tightness
condition is satisfied. They also obtain geometric convergence in the
total variation distance. Nevertheless, there are many simple processes
with switching which do not verify any hypoellipticity condition. To
illustrate this fact, let us consider the simple example of \cite
{MLMZ-Quant}. Let $(X,I)$ be the Markov process on $\R^2 \times\{
-1,1\}$ generated by
%
%
%e1.1 #&#
\begin{equation}
\label{eq:exemple-intro} A f(x,i) = - \bigl(x-(i,0)\bigr) \cdot\nabla_x f(x,i)
+ \bigl(f(x,-i) - f(x,i) \bigr).
\end{equation}
This process is ergodic and the first marginal $\pi$ of its invariant
measure is supported on $\R\times\{0\}$. This suggests that, in
general, the law of the
process does not converge to its invariant measure in the total
variation distance. However, it was proved
in \cite{MLMZ-Quant} that it converges in a certain Wasserstein
distance. Let
us recall that the $p${th} Wasserstein distance $\mathcal
{W}^{(p)}$, with $p\geq1$, on a Polish space $(E,d)$ is defined by
\[
\mathcal{W}^{(p)}_d ( \mu_1,\mu_2 )
= \inf_{X_1,X_2} \E \bigl[d(X_1,X_2)^p
\bigr]^{1/p},
\]
for any two probability measures $\mu_1$, $\mu_2$ on $E$, where the
infimum is taken over all pairs of $E$-valued random variables $X_1$,
$X_2$ with respective laws $\mu_1$, $\mu_2$. When $p=1$, we set
$\mathcal{W}_d =\mathcal{W}^{(1)}_d$. The Kantorovich--Rubinstein
duality (\cite{Vil09}, Theorem~5.10) shows that one also has
\[
\mathcal{W}_d ( \mu_1,\mu_2 ) = \sup
_{f
\in\mathrm{Lip}_1} \biggl(\int_E f \,\mathrm{d}
\mu_1 - \int_E f\, \mathrm{d}\mu_2
\biggr),
\]
where $f\dvtx E \mapsto\R$ is in $\mathrm{Lip}_1$ if and only if it
is a $1$-Lipschitz function, namely
\[
\forall x,y \in E,\qquad\bigl |f(x)- f(y)\bigr | \leq d(x,y).
\]
The total variation distance $d_\TV$ can be viewed as the Wasserstein
distance associated to the trivial distance function, namely
\[
d_\TV(\mu_1,\mu_2)= \inf
_{X_1,X_2} \p (X_1 \neq X_2 ) =
\frac{1}{2} \sup_{\Vert f \Vert_\infty\leq1} \biggl( \int_E
f \,\mathrm{d}\mu_1 - \int_E f \,\mathrm{d}
\mu_2 \biggr),
\]
where the infimum is again taken over all random variables $X_1$, $X_2$
with laws $\mu_1$, $\mu_2$. In the present article, we will give
convergence criteria for a general class of switching Markov processes.
These processes are built from the following ingredients:
\begin{itemize}
\item a Polish space $(E,d)$ and a finite set $F$;
\item a family $(Z^{(n)})_{n\in F}$ of $E$-valued strong Markov
processes represented by their semigroups $(P^{(n)})_{n\in F}$, or
equivalently by their generators $(\mathcal{L}^{(n)})_{n\in F}$ with
domains $(\mathcal{D}^{(n)})_{n\in F}$;
%Throughout the whole article, we will assume that each underlying
%dynamics is stochastically continuous; namely
%$$
%$$
%
\item a family $(a(\cdot, i, j))_{i,j\in F}$ of non-negative functions
on $E$.
\end{itemize}

We are interested by the process $(\mathbf{X}_t)_{t\geq0}= (X_t,
I_t)_{t\geq0}$, defined on $\mathbf{E}=E\times F$, which jumps
between these dynamics. Roughly speaking, $X_t$ behaves like
$Z^{(I_t)}_t$ as long as $I$ does not jump. The process $I$ is discrete
and jumps at a rate given by $a$. More precisely, the dynamics of
$(\mathbf{X}_t)_{t\geq0}$ is as follows:
\begin{itemize}
\item Given a starting point $(x,i) \in E \times F$, we take for
$Z^{(i)}$ an instance as above with
initial condition $Z^{(i)}_0 = x$. The initial conditions for $Z^{(j)}$
with $j \neq i$ are irrelevant.
\item The discrete component $I$ is constant and equal to $i$ until the
time $T=\min_{j\in F} T_j$, where $(T_j)_{j\geq0}$ is a family of
random variables that are conditionally independent given $Z^{(i)}$
and that verify
\[
\forall j\in F,\qquad\p ( T_j > t | \mathcal{F}_t ) =
\exp \biggl(- \int_0^t a \bigl(Z^{(i)}_s,i,j
\bigr) \,\mathrm{d}s \biggr),
\]
where $\mathcal{F}_t = \sigma\{ Z^{(i)}_s | s\leq t \}$.
\item For all $t\in[0, T)$, we then set $X_t= Z^{(i)}_t$ and $I_t = i$.
\item At time $T$, there exists a unique $j\in F$ such that $T=T_j$ and
we set $I_{T}=j$ and $X_{T}=X_{T-}$.
\item We take $(X_{T}, I_{T})$ as a new starting point at time $T$.
\end{itemize}
Let us make a few remarks about this construction. First, this
algorithm guarantees the
existence of our process under the condition that there is no explosion
in the switching rate.
In other words, our construction is global as long as $I$ only switches
value finitely many time
in any finite time interval. Assumption~\ref{hyp:a-reg} below will be
sufficient to guarantee this non-explosion. Also note that, in general,
$X$ and $I$ are not Markov
processes by themselves, contrary to $\mathbf{X}$. Nevertheless, we
have that $I$ is a Markov process if
$a$ does not depend on its first component.
The construction given above shows that, provided that there is no explosion,
the infinitesimal generator of $\mathbf{X}$ is given by
%
%
%e1.2 #&#
\begin{equation}
\label{eq:generateur} \mathbf{L} f(x,i) = \mathcal{L}^{(i)} f(x,i) + \sum
_{j \in F} a(x,i,j) \bigl( f(x,j) -f(x,i) \bigr),
\end{equation}
for any bounded function $f$ such that $f(\cdot,i)$ belongs to
$\mathcal{D}^{(i)}$ for every $i \in F$. We will denote by $(\mathbf
{P}_t)_{t\geq0}$ the semigroup of $\mathbf{X}$. To guarantee the
existence of our process, we will consider the following natural assumption:

%
%as1.1 #&#
\begin{hypo}[(Regularity of the jumps rates)]\label{hyp:a-reg}
The following boundedness condition is verified:
\[
\bar{a} = \sup_{x\in E} \sup_{i \in F} \sum
_{j \in F} a(x,i,j) < + \infty,
\]
and the following Lipschitz condition is also verified:
\[
\sup_{i \in F} \sum_{j\in F} \bigl |a(x,i,j)
- a(x,i,j)\bigr | \leq\kappa d(x,y),
\]
for some $\kappa>0$.
\end{hypo}

We will also assume the following hypothesis to guarantee the
recurrence of $I$:
%

%as1.2 #&#
\begin{hypo}[(Recurrence assumption)]\label{hyp:rec}
The matrix $(\underline{a}(i,j))_{i,j\in F}$ defined by
\[
\underline{a}(i,j) = \inf_{x\in E} a(x,i,j),
\]
yields the transition rates of an irreducible and positive recurrent
Markov chain.
\end{hypo}

With these two assumptions, we are able to get exponential stability in
two situations. The first situation
is one where each underlying dynamics does on average yield a
contraction in some Wasserstein distance, but
no regularising assumption is made.
The second situation is the opposite, where we replace the contraction
by a suitable regularising property.

%s1.1 #&#
\subsection{Two criteria without hypoellipticity assumption}

In this section, we assume that we have some information on the
Lipschitz contraction (or expansion)
of our underlying processes:

%
%as1.3 #&#
\begin{hypo}[(Lipschitz contraction)]\label{hypo:Curv}
For each $i \in F$, there exists $\rho(i) \in\R$ such that
%
%
%e1.3 #&#
\begin{equation}
\label{eq:curv} \forall t\geq0,\qquad\mathcal{W}_d \bigl( \mu
P^{(i)}_t, \nu P^{(i)}_t \bigr) \leq
\mathrm{e}^{ - \rho(i) t} \mathcal{W}_d ( \mu, \nu ),
\end{equation}
for any two probability measures $\mu, \nu$. Furthermore there exist
$x_0\in E$ and $t_{x_0}>0$ such that if $V_{x_0}\dvtx x\mapsto
d(x,x_0)$ then
\[
\sup_{t \in[0, t_{x_0}]} P_t V_{x_0}
(x_0) < + \infty.
\]
\end{hypo}

In the previous assumption, given a semigroup $(P_t)_{t\geq0}$, we
used the notation $\mu P_t$ to denote the measure defined by
\[
(\mu P_t) f= \int P_t f \,\mathrm{d}\mu,
\]
if $\mu= \delta_x$, for some $x$, then in this work, we also use the
notation $\delta_x P_t (\mathrm{d}y) = P_t(x, \mathrm{d}y)$.

To verify equation \eqref{eq:curv} is not much of a restriction
because we do not assume that $\rho(i)>0$. The best constant in this
inequality is called the Wasserstein curvature in \cite{J07,J09} and
the coarse Ricci curvature in \cite{O09,O10}, since it is heavily
related to the geometry of the underlying space as illustrated in \cite
{RS05}, Theorem~2. If $\rho(i) > 0$, then we can deduce some
properties like geometric ergodicity, a Poincar\'{e} inequality or some
concentration inequalities \cite{Clo12,J07,J09,HSV,O10}. A trivial
bound on $\rho(i)$ is given in the special case of diffusion processes
in Section~\ref{sect:dif}.

The bound \eref{eq:curv} is quite stringent since, if $\rho(i) > 0$,
it implies that there
is some Wasserstein contraction for \textit{every} $t>0$ and not just
for sufficiently long times.
This is essentially equivalent to the existence of a Markovian coupling
between two instances
$X_t$ and $Y_t$ of the Markov process with generator $\mathcal
{L}^{(i)}$ such that
$\E d(X_t, Y_t) \le\mathrm{e}^{-\rho t} d(X_0, Y_0)$.

In principle, this condition could be slightly relaxed by the addition
of a proportionality
constant $C_i$, provided that one assumes that the
switching rate of the process is sufficiently slow. This ensures
that, most of the time, it spends a sufficiently long time in any one
state for this proportionality
constant not to play a large role.

One could also imagine allowing for jumps of the component in $E$ at
the switching times, and
this would lead to a similar difficulty.

In the same way, the distance $d$ appearing in Assumption~\ref
{hypo:Curv} is the \textit{same}
for every $i$ and that it does not allow for a constant prefactor in
the right-hand side of \eqref{eq:curv}. This may seem like a very
strong assumption since usual convergence theorems, like Harris'
theorem, do not give this kind of bound. We will see however in
Section~\ref{sect:exemple} an example which illustrates that there is no
obvious way in general to
weaken this condition. The intuitive reason why this is so is that if
the process switches
rapidly, then it is crucial to have some local information (small
times) and not only
global information (large times) on the behaviour of each underlying dynamics.\vadjust{\goodbreak}

We have now presented all the assumptions required
to state our main results.
The first one describes the simplest situation, that is when $a$ does
not depend on its first component:

%
%th1.4 #&#
\begin{theo}[(Wasserstein exponential ergodicity in the constant case)]
\label{th:W-cst}
Under Assumptions~\ref{hyp:a-reg}, \ref{hyp:rec} and~\ref
{hypo:Curv}, if $a(x,i,j)$ does not depend on $x$
and the Markov process $I$ has an invariant probability measure $\nu$ verifying
\[
\sum_{i\in F} \nu(i) \rho(i) > 0,
\]
then there exist a probability measure $\boldsymbol{\pi} $, some
constants $C, \lambda>0$ and $q\in(0,1]$ such that
\[
\forall t\geq0,\qquad\mathcal{W}_{\mathbf{d}} (\delta _{\mathbf
{y}_0}
\mathbf{P}_t, \boldsymbol{\pi} ) \leq C \mathrm {e}^{-\lambda t}
\biggl(1+\sum_{i\in F} \int_E
d(y_0,x)^q \boldsymbol{\pi} (\mathrm{d}x,i) \biggr),
\]
for every $\mathbf{y}_0 = (y_0,j_0) \in\mathbf{E}$, where the
distance $\mathbf{d}$, on $\mathbf{E}$, is defined by
%
%
%e1.4 #&#
\begin{equation}
\label{eq:dgras} \mathbf{d}(\mathbf{x},\mathbf{y}) = \one_{i\neq j} +
\one_{i = j} \bigl(1 \wedge d^q(x,y)\bigr),
\end{equation}
for every $\mathbf{x}=(x,i)$, $\mathbf{y}=(y,j)$ belonging to
$\mathbf{E}$.
\end{theo}

%NB: Attention a la cond initiale, dois je supp ke c tjr vrai ou
%seulement pr un $i$?\\
This statement is not surprising: it states that if the process
contracts in mean, then it converges exponentially to an invariant
distribution. The conditions are rather sharp as will be illustrated in
Section~\ref{sect:exemple}. In particular, we recover \cite
{MLMZ-Quant}, Theorem~1.10, and this (slight) generalisation could be deduced
from the argument given there. Using H\"{o}lder's inequality, we can
also deduce convergence in the $p${th} Wasserstein distance
$\mathcal{W}^{(p)}_{\mathbf{d}}$ with $p\geq1$ provided that
$\mathbf{X}$ satisfies a moment condition.

We provided Theorem~\ref{th:W-cst} and its proof for sake of
completeness and for a better understanding of the more complicated
case, where $a$ is allowed to depend on its first argument. In this
situation, our main result reads as follows.

%
%th1.5 #&#
\begin{theo}[(Wasserstein exponential ergodicity with an on--off type
criterion)]\label{th:W-noncst}
Suppose that Assumptions~\ref{hyp:a-reg}, \ref{hyp:rec}, and~\ref
{hypo:Curv} hold and set
\begin{eqnarray*}
F_0&=& \bigl\{i\in F | \rho(i) > 0 \bigr\} \quad\mbox{and}\quad
F_1 = \bigl\{i\in F | \rho(i) \leq0 \bigr\},
\\
\rho_0 &=& \min_{i\in F_0} \rho(i) > 0 \quad
\mbox{and}\quad\rho _1 = \min_{i\in F_1} \rho(i) \leq0,
\\
a_0 &=& \max_{i\in F_0} \sup_{x\in E}
\sum_{j\in F_1} a(x,i,j) \quad \mbox{and}\quad
a_1=\min_{i\in F_1} \inf_{x\in E} \sum
_{j\in F_0} a(x,i,j).
\end{eqnarray*}
If
\[
\rho_0 a_1 + \rho_1 a_0>0,
\]
then there exist a probability measure $\boldsymbol{\pi} $, some
constants $C, \lambda>0$ and $q\in(0,1]$ such that
\[
\forall t\geq0,\qquad\mathcal{W}_{\mathbf{d}} (\delta _{\mathbf
{y}_0}
\mathbf{P}_t, \boldsymbol{\pi} ) \leq C \mathrm {e}^{-\lambda t}
\biggl(1+\sum_{i\in F} \int_E
d(y_0,x)^q \boldsymbol{\pi} (\mathrm{d}x,i) \biggr),
\]
for every $\mathbf{y}_0 = (y_0,j_0) \in\mathbf{E}$, where the
distance $\mathbf{d}$
on $\mathbf{E}$ is again given by \eref{eq:dgras}.
\end{theo}

With this result, we not only recover \cite{MLMZ-Quant}, Theorem~1.15,
but we extend it significantly. In our case, the underlying dynamics are
not necessarily deterministic and do not need to be strictly
contracting in
a Wasserstein distance. One drawback is that the constants $\lambda$
and $C$ are
much less explicit. This theorem is a direct consequence of the more general
Theorem~\ref{th:W-noncst2} below. These two theorems are our main results
and, contrary to Theorem~\ref{th:W-cst}, it seems that they cannot be deduced
directly from the approach of \cite{MLMZ-Quant}.

%s1.2 #&#
\subsection{Two criteria with hypoellipticity assumption}

In the previous subsection, we have supposed that some of the
underlying dynamics contract
at sufficiently high rate in a Wasserstein distance. This is of course
not a necessary condition
for geometric ergodicity in general. Using some arguments of the proof
of Theorem~\ref{th:W-cst} and
Theorem~\ref{th:W-noncst}, we can deduce a different criterion which
uses instead a Lyapunov-type argument to prove that $\mathbf{X}$
converges. We begin by stating an assumption similar to Assumption~\ref
{hypo:Curv}:
%

%as1.6 #&#
\begin{hypo}[(Existence of a Lyapunov function)]\label{hyp:Lyap}
There exist $K\geq0$, a function $V \geq0$, and for every $i\in F$
there exists $\lambda(i)\in\R$ such that
%
%
%e1.5 #&#
\begin{equation}
\label{eq:Lyap-underlying} \forall t\geq0,\forall x\in E,\qquad P^{(i)}_t
V(x) \leq\mathrm {e}^{-\lambda(i) t} V(x) + K.
\end{equation}
\end{hypo}

Note again that we have not supposed that $\lambda(i)>0$. One way to
prove this kind of bound is to use the classical drift condition on the
generator (see \eqref{eq:CD} below). With this assumption we are able
to prove the following result, where
the definition of a ``small set'' will be recalled in Definition~\ref
{def:small}
below.

%
%th1.7 #&#
\begin{theo}[(Exponential ergodicity in the constant case)]
\label{th:dtv-cst}
Suppose that Assumptions~\ref{hyp:a-reg}, \ref{hyp:rec} and~\ref
{hyp:Lyap} hold, that $a(x,i,j)$ does not depend on $x$
and that $I$ has an invariant probability measure $\nu$ verifying
\[
\sum_{i\in F} \nu(i) \lambda(i) > 0.
\]
If there exists $i_0\in F$ and $t_0\geq0$ such that the sublevel sets
$\{x \in E | V(x) \leq K \}$ of $V$ are small for $P^{(i_0)}_t$ for
every $K>0$ and $t\geq t_0$,
then there exist a probability measure $\boldsymbol{\pi} $ and two
constants $C, \lambda>0$ such that
\[
\forall t\geq0,\qquad d_\TV (\delta_{\mathbf{x}}
\mathbf{P}_t, \boldsymbol{\pi} ) \leq C \mathrm{e}^{-\lambda t}
\bigl(1+V(x)\bigr),
\]
for every $\mathbf{x}=(x,i) \in\mathbf{E}$.
\end{theo}

We give also the result analogous to Theorem~\ref{th:W-noncst}.

%
%th1.8 #&#
\begin{theo}[(Exponential ergodicity with an on--off type criterion)]
\label{th:dtv-noncst}
Suppose that Assumptions~\ref{hyp:a-reg}, \ref{hyp:rec}, \ref
{hypo:Curv} hold and set
\begin{eqnarray*}
F_0&=& \bigl\{i\in F | \lambda(i) > 0 \bigr\} \quad\mbox{and}\quad
F_1 = \bigl\{ i\in F | \lambda(i) \leq0 \bigr\},
\\
\lambda_0 &=& \min_{i\in F_0} \lambda(i) > 0 \quad
\mbox{and}\quad \lambda_1 = \min_{i\in F_1} \lambda(i)
\leq0,
\\
a_0 &=& \max_{i\in F_0} \sup_{x\in E}
\sum_{j\in F_1} a(x,i,j) \quad \mbox{and}\quad
a_1=\min_{i\in F_1} \inf_{x\in E} \sum
_{j\in F_0} a(x,i,j).
\end{eqnarray*}
If
\[
\lambda_0 a_1 + \lambda_1
a_0>0,
\]
and there exists $i_0\in F$ and $t_0 \geq0$ such that all sublevel sets
of $V$ are small for $P^{(i_0)}_t$, for every $t\geq t_0$, then there
exist a probability measure $\boldsymbol{\pi} $ and two constants $C,
\lambda>0$ such that
\[
\forall t \geq0,\qquad d_\TV (\delta_{\mathbf{x}}
\mathbf{P}_t, \boldsymbol{\pi} ) \leq C \mathrm{e}^{-\lambda t}
\bigl(1+V(x)\bigr),
\]
for every $\mathbf{x}=(x,i) \in\mathbf{E}$.
\end{theo}

Note that in general it is not necessary to assume that sublevel sets
of $V$ are small for
any single one of the underlying dynamics. For example, using the
results of \cite{BH12,MLMZ-Horm},
Section~\ref{sect:PDMP} gives results analogous to the two previous
theorems, in the
special case of piecewise deterministic Markov processes where the only
small sets for the
underlying dynamics consist of single points.

The remainder of the paper is organised as follows. The proofs of our
four main theorems are split over two sections: Section~\ref{sect:cst}
deals with the proof of Theorem~\ref{th:W-cst} and Theorem~\ref
{th:dtv-cst}. In Section~\ref{sect:non-cst}, we begin by giving a more
general assumption in the non-constant case than our on--off criterion.
Then, we introduce a weak form of Harris' theorem that we will use to
prove Theorem~\ref{th:W-noncst}. The proof of this theorem is then
decomposed in such a way to verify each point of the weak Harris'
theorem. Section~\ref{sect:dif} gives sufficient conditions to verify
our main assumption in the special case of diffusion processes. The
section which follows deals with the special case of switching
dynamical system. We conclude with Section~\ref{sect:exemple}, where
we give some very simple examples illustrating the sharpness of our conditions.

%s2 #&#
\section{Constant jump rates}
\label{sect:cst}

In this section, we begin by proving that under Assumptions~\ref
{hypo:Curv} or~\ref{hyp:Lyap}, the process $\mathbf{X}$ cannot wander
off to infinity, that is, its semigroup possesses a Lyapunov function.
We then prove Theorems~\ref{th:W-cst} and~\ref{th:dtv-cst} using a
similar argument to \cite{MLMZ-Quant} for the first one and Harris'
theorem for the second one.

%s2.1 #&#
\subsection{Construction of a Lyapunov function}

We begin by recalling the definition of a Lyapunov function
%

%de2.1 #&#
\begin{defi}[(Lyapunov function)]
\label{def:lyap}
A Lyapunov function for a Markov semigroup $(P_t)_{t\geq0}$ over a
Polish space $(X,d_X)$ is a function $V\dvtx X\mapsto[0,\infty]$ such
that $V$ is integrable with respect to $P_t (x, \cdot)$ for every
$x\in X$ and $t> 0$ and such that there exist constants $C_V, \gamma,
K_V>0$ verifying
%
%
%e2.1 #&#
\begin{equation}
\label{eq:def-Lyap} P_t V(x) = \int_X V(y)
P_t(x,\mathrm{d}y) \leq C_V \mathrm {e}^{-\gamma t}
V(x) + K_V,
\end{equation}
for every $x \in X$ and $t \geq0$.
\end{defi}

A well-known sufficient condition for finding a Lyapunov function is
the following drift condition:
%
%
%e2.2 #&#
\begin{equation}
\label{eq:CD} \mathcal{L} V \leq-\gamma V + C,
\end{equation}
where $\mathcal{L}$ is the generator of the semigroup $(P_t)_{t\geq
0}$. The condition \eqref{eq:CD} implies a bound like \eqref
{eq:Lyap-underlying} and is clearly stronger than \eqref{eq:def-Lyap}.
In general, our switching Markov process $\mathbf{X}$ may not verify
the drift condition \eqref{eq:CD} but,
in Lemmas~\ref{lem:Lyap-cst} and~\ref{lem:Lyap-noncst}, we give a
sharp condition under which it verifies \eqref{eq:def-Lyap}. In this
section, we first prove that a Wasserstein contraction as in Assumption~\ref{hypo:Curv} implies the existence of a Lyapunov-type function as
in Assumption~\ref{hyp:Lyap}. Then, we will prove that Assumption~\ref
{hyp:Lyap} implies the existence of a Lyapunov function for $\mathbf{X}$.

%
%le2.2 #&#
\begin{lem}[(Wasserstein contraction implies the existence of a
Lyapunov-type function)]
\label{lem:Curv-Lyap}
Let $(P_t)_{t\geq0}$ be the semigroup of a Markov process, on a Polish
space $(X,d_X)$, such that there exists $\lambda\in\R^*$ verifying
%
%
%e2.3 #&#
\begin{equation}
\label{eq:contraction-lyap} \mathcal{W}_{d_X} (\delta_x
P_t, \delta_y P_t) \leq\mathrm
{e}^{-\lambda t} d_X (x,y),
\end{equation}
for every $x,y\in X$ and $t\geq0$. If there exist $x_0\in X$ and
$t_{x_0}>0$ such that the function $V_{x_0}\dvtx x\mapsto d(x,x_0)$ verifies
%
%
%e2.4 #&#
\begin{equation}
\label{e:boundV} \sup_{t \in[0, t_{x_0}]} P_t V_{x_0}
(x_0) < + \infty,
\end{equation}
then there exist $C_1,C_2>0$ such that
%
%
%e2.5 #&#
\begin{equation}
\label{eq:lyap-pondere} P_t V_{x_0}(x) \leq\mathrm{e}^{- \lambda t }
\bigl(V_{x_0}(x) + C_1\bigr) + C_2,
\end{equation}
for every $x\in X$ and $t\geq0$.
\end{lem}

\begin{pf}
Note first that the bound \eqref{eq:contraction-lyap} is equivalent to
the bound
%
%
%e2.6 #&#
\begin{equation}
\label{eq:contraction-lyap2} \mathcal{W}_{d_X} (\mu P_t, \nu
P_t) \leq \mathrm{e}^{-\lambda t} \mathcal {W}_{d_X} (\mu, \nu),
\end{equation}
for \textit{every} probability measure $\mu$ and $\nu$, as a
consequence of the
bound $\mathcal{W}_{d_X} (\mu P_t, \nu P_t) \leq\int\mathcal
{W}_{d_X} (\delta_x P_t, \delta_y P_t) \pi(\mathrm{d}x,\mathrm
{d}y)$ which follows
immediately from the definitions
and is true for any measure $\pi$ with marginals $\mu$ and $\nu$.

For any $t\geq t_{x_0}$ and $n\geq0$, it then follows from \eref
{eq:contraction-lyap2} that
\begin{eqnarray*}
P_t V_{x_0} (x_0) &=&\mathcal{W}_{d_X}
(\delta_{x_0} P_t, \delta_{x_0}) \leq\sum
_{k=0}^{n-1} \mathcal{W}_{d_X} (
\delta_{x_0} P_{(k+1){t}/{n}}, \delta_{x_0} P_{k{t}/{n}}
)
\\
&\le&\sum_{k=0}^{n-1} \mathrm{e}^{-\lambda k t/n}
\mathcal{W}_{d_X} (\delta _{x_0}P_{{t}/{n}},
\delta_{x_0})\leq\frac{\mathrm{e}^{-\lambda t}
-1}{\mathrm{e}^{-\lambda t/n} -1} P_{t/n} V_{x_0}
(x_0).
\end{eqnarray*}
Taking $n= \lfloor t/t_{x_0} \rfloor+1$, where $\lfloor\lambda
\rfloor$ denotes the
integer part of $\lambda$,
we conclude that
\[
P_t V_{x_0} (x_0) \leq\bigl(
\mathrm{e}^{-\lambda t} + 1\bigr) C',\qquad C'= \sup
_{u \in [t_{x_0}/2, t_{x_0} ]} \frac{P_{u}
V_{x_0} (x_0)}{|\mathrm{e}^{-\lambda u} -1|},
\]
which is finite by \eref{e:boundV}. Finally, for every $x\in X$ and
$t\geq0$, we have
\begin{eqnarray*}
P_t V_{x_0}(x) &=& \mathcal{W}_{d_X} (
\delta_{x} P_{t}, \delta_{x_0}) \leq
\mathcal{W}_{d_X}(\delta_{x} P_{t},
\delta_{x_0} P_t) + \mathcal{W}_{d_X}(
\delta_{x_0} P_{t}, \delta_{x_0})
\\
&\leq& \mathrm{e}^{-\lambda t} V_{x_0}(x) + \bigl(
\mathrm{e}^{-\lambda t} + 1\bigr)C',
\end{eqnarray*}
thus concluding the proof.
\end{pf}

%
%re2.3 #&#
\begin{Rq}
The point of this lemma is to also allow for negative values of
$\lambda$. When $\lambda> 0$, then
it is immediate that $P_t$ admits a unique invariant measure and
exhibits geometric ergodicity.
\end{Rq}

%
%re2.4 #&#
\begin{Rq}
If $V_{x_0}$ is in the domain of the generator $\mathcal{L}$ of
$(P_t)_{t\geq0}$, then we have
\[
\forall t\geq0,\qquad P_t V_{x_0} (x_0) \leq
\frac{\mathrm{e}^{-\lambda t} -1}{\mathrm{e}^{-\lambda t/n} -1} P_{t/n} V_{x_0} (x_0),
\]
for some $n\geq1$. Now, taking the limit $n\rightarrow+ \infty$, we
deduce the following bound:
\[
\mathcal{W}_{d_X} (\delta_{x_0} P_{t},
\delta_{x_0}) \leq\frac
{\mathrm{e}^{-\lambda t} -1}{-\lambda} \mathcal{L} V(x_0).
\]
Finally, for every $x\in X$, we have
\begin{eqnarray*}
P_t V(x) &=& \mathcal{W}_{d_X} (\delta_{x}
P_{t}, \delta_{x_0}) \leq\mathcal{W}_{d_X}(
\delta_{x} P_{t}, \delta_{x_0} P_t) +
\mathcal{W}_{d_X}(\delta_{x_0} P_{t},
\delta_{x_0})
\\
&\leq& \mathrm{e}^{-\lambda t} V(x) + \frac{\mathrm{e}^{-\lambda t}
-1}{-\lambda} \mathcal{L}
V(x_0).
\end{eqnarray*}
However, $V_{x_0}$ does not belong to the domain of the generator in
general, as can be seen
already in the example of simple Brownian motion.
\end{Rq}

%
%re2.5 #&#
\begin{Rq}[(The special case $\lambda= 0$)]
The assumption $\lambda\neq0$ is required for our
conclusion to hold. Indeed, if $(B_t)_{t\geq0}$ is a Brownian motion then
\[
\lim_{t \rightarrow+ \infty} \E \bigl[ | B_t | \bigr] = + \infty,
\]
and inequality \eqref{eq:lyap-pondere} does not hold. Instead, it is
straightforward to follow the argument of the proof to show that if
$\lambda= 0$ in
Lemma~\ref{lem:Curv-Lyap},
then one has the bound
\[
P_t V_{x_0}(x) \leq V_{x_0}(x) + Ct,
\]
for some fixed constant $C>0$, every $x\in E$, and every $t\geq0$.
\end{Rq}

%
%re2.6 #&#
\begin{Rq}
By Lemma~\ref{lem:Curv-Lyap}, Assumption~\ref{hypo:Curv} implies
Assumption~\ref{hyp:Lyap} with $\lambda= \rho$
and $V=V_{x_0}$ as long as one has $\rho(i) \neq0$
for every $i$.
In general, without any assumption on $\rho$, it does of course imply
Assumption~\ref{hyp:Lyap} for any function $\lambda$ with
$\lambda(i) < \rho(i)$, which is sufficient for our needs.
\end{Rq}

We now show that if Assumption~\ref{hyp:Lyap} holds and the mean of
$(\lambda(i))_{i\in F}$ is positive, then $\mathbf{X}$ admits a
Lyapunov function. As in \cite{MLMZ-Quant}, this result is obtained as
a consequence of the following lemma:

%
%le2.7 #&#
\begin{lem}
\label{lem:Per-Frob}
Let $(K_t)_{t \ge0}$ be a continuous-time Markov chain on a finite set
$S$, and assume that it is irreducible and positive recurrent with
invariant measure $\nu_K$. If $\alpha\dvtx S \to\R$ is a function verifying
\[
\sum_{n\in S} \nu_K(n) \alpha(n)>0,
\]
then there exist $C,c,\eta>0$ and $p\in(0,1]$ such that
\[
c \mathrm{e}^{-\eta t} \leq\E \bigl[ \mathrm{e}^{-\int_0^t p
\alpha(K_s) \,\mathrm{d}s} \bigr] \leq
C \mathrm{e}^{- \eta t},
\]
for any initial condition $K_0$ and every $t\geq0$.
\end{lem}

\begin{pf}
It is a consequence of Perron--Frobenius theorem and the study of
eigenvalues. See \cite{BGM}, Proposition~4.1, and \cite{BGM}, Proposition~4.2, for further details.
\end{pf}

Now we are able to prove that $\mathbf{P}$ possesses a Lyapunov
function in the case
where the switching rates do not depend on the location of the process.
%

%le2.8 #&#
\begin{lem}
\label{lem:Lyap-cst}
Under Assumptions~\ref{hyp:a-reg}, \ref{hyp:rec} and~\ref{hyp:Lyap},
if $a(x,i,j)$ does not depend on $x$ and $I$ has an invariant measure
$\nu$ satisfying
\[
\sum_{i\in F} \lambda(i) \nu(i)>0,
\]
then there exist $C_V, K_V, \lambda_V>0$ and $q\in(0,1]$ such that
\[
\forall t\geq0, \forall x\in E, \qquad\mathbf{P}_t V^{q}
(x,i) \leq C_V \mathrm{e}^{- \lambda_V t} V^{q} (x) +
K_V.
\]
\end{lem}

In the previous lemma, we used a slight abuse of notation. Indeed, if
$f$ is a function defined on $E$, we also denote by $f$ the mapping
$(x,i) \mapsto f(x)$ on $\mathbf{E}$.
\begin{pf} First, Jensen's inequality gives this weaker form of
\eqref{eq:Lyap-underlying}:
\[
P^{(i)}_t \bigl(V^{q}\bigr) (x) \leq
\mathrm{e}^{- q \lambda(i) t} V^{q}(x) + K^{q},
\]
for every $q\in(0,1]$. Now, for all $t\geq0$ and $(x,i) \in\mathbf
{E}$, a straightforward recurrence gives
\begin{eqnarray*}
\mathbf{P}_t V^{q} (x,i) &=& \E \bigl[
P^{(I_{T_{N_t}})}_{t -T_{N_t}} \circ P^{(I_{T_{N_t}
-1})}_{T_{N_t} -T_{N_t-1}} \circ
\cdots\circ P^{(I_{0})}_{T_1- T_0} \bigl(V^{q}\bigr) (x)
\bigr]
\\
&\leq&\E \bigl[ \mathrm{e}^{-\int_0^t q \lambda(I_s) \,\mathrm
{d}s } \bigr] V^q(x) +
K^q \sum_{n \geq0} \E \bigl[
\mathrm{e}^{-q \int_0^{T_n} \lambda
(I_s) \,\mathrm{d}s} \bigr],
\end{eqnarray*}
where $(T_k)_{k\geq0}$ is the sequence of jump times of $I$, with
$T_0=0$, and $N_t$ the number of jumps before $t$. By Lemma~\ref
{lem:Per-Frob}, there exist $C>0, \eta>0$ and $q \in(0,1]$ such that
\[
\E \bigl[ \mathrm{e}^{-\int_0^t q \lambda(I_s) \,\mathrm{d}s
} \bigr] \leq C \mathrm{e}^{-\eta t}.
\]
Furthermore, one can show that $T_n$ is of order $n$ and that
\[
K_V= K^q \sum_{n \geq0} \E
\bigl[ \mathrm{e}^{-q \int_0^{T_n}
\lambda(I_s)
\,\mathrm{d}s} \bigr] \lesssim K^q \sum
_{n \geq0} \mathrm {e}^{-\varepsilon n} <+ \infty,
\]
for some $\varepsilon>0$. We do not detail this argument now, but we
will prove it in the
slightly more difficult context of non-constant rate $a$ in
Lemma~\ref{lem:Lyap-noncst}. This concludes the proof.
\end{pf}

%
%re2.9 #&#
\begin{Rq}[(On the assumption that $F$ is finite)]
It is natural to extend our results to the case where $F$ is countably infinite.
Obviously, we then have to add the assumption that $I$ is positive
recurrent, but this is not enough.
Indeed, if for each $i\in F$, $C_1(i)$ and $C_2(i)$ denote the
constants $C_1, C_2$, appearing
in Lemma~\ref{lem:Curv-Lyap} applied on $Z^{(i)}$, then we should
furthermore assume that
\[
\sup_{i \in F} \bigl(C_1(i) + C_2(i)
\bigr) < + \infty,
\]
for the argument to go through.
\end{Rq}

%s2.2 #&#
\subsection{Proof of Theorem \texorpdfstring{\protect\ref{th:W-cst}}{1.4}}

The proof of this result is obtained by a coupling construction. We
first give
a description of this construction and we then turn to the proof itself.
Throughout this section, we make the standing assumption that
the hypotheses of Theorem~\ref{th:W-cst} hold. In particular, $I$ is
an ergodic
finite-state Markov chain.

Let $\mathbf{x}=(x,i)$ and $\mathbf{y}=(y,j)$ be two points of
$\mathbf{E}$, we will build a coupling $(\mathbf{X},\mathbf{Y})$,
starting from $(\mathbf{x},\mathbf{y})$, such that each component is
an instance
of the Markov process generated by $\mathbf{L}$, and such that
the distance $\mathbf{d}(\mathbf{X}_t, \mathbf{Y}_t)$ decreases to $0$
at exponential rate.
From now on, we fix the starting points of our coupling $\mathbf
{x}=(x,i)$, $\mathbf{y}=(y,j)$. The processes $(\mathbf{X}_t)_{t\geq
0}=(X_t,I_t)_{t\geq0}$ and $(\mathbf{Y}_t)_{t\geq0}=(Y_t,J_t)_{t\geq
0}$ are
then constructed as follows:
\begin{itemize}
\item First, we run both processes independently until the first hitting
time $T_c = \inf\{ t\geq0 | I_t= J_t \}$ of the two components $I$
and $J$.
In case we start with an initial condition such that $i=j$,
then we simply set $T_c=0$.
\item For times $s\geq T_c$, we set $I_s=J_s$ and we couple $X$ and $Y$
in such a way that
\[
\forall k \geq0, \qquad\E \bigl[ d(X_{S_k}, Y_{S_k}) |
\mathcal {F}_{S_{k-1}} \bigr] \leq\mathrm{e}^{- \rho(I_{S_{k-1}}) (S_{k} -
S_{k-1}) }
d(X_{S_{k-1}}, Y_{S_{k-1}}),
\]
where $(T_k)_{k\geq0}$ is the sequence of jumps times of $I$, $S_k=T_k
\wedge t$ and $(\mathcal{F}_s)_{s\geq0}$ is the natural filtration
associated to $(\mathbf{X},\mathbf{Y})$.
\end{itemize}
The existence of a coupling satisfying the second point is an immediate
consequence of
Assumption~\ref{hypo:Curv}.

%We set. If $I_t \neq J_t$ then we consider that $\mathbf{X}$ and $
%we have to build a coupling in such a way that the distance between
%$X$ and $Y$ will decrease in mean. It is not a real problem \ldots

\begin{pf*}{Proof of Theorem~\ref{th:W-cst}}
Recall first that if $I$ and $J$ are two independent finite-state
Markov chains
with transition rate $a$ as in the statement of Theorem~\ref{th:W-cst},
then there exist constants $C_c$, $\theta_c>0$ such that
%
%
%e2.7 #&#
\begin{equation}
\label{eq:coalescent} \forall t\geq0,\qquad\p ( T_c > t ) \leq
C_c \mathrm {e}^{-\theta_c t},
\end{equation}
for any two initial conditions $I_0$ and $J_0$.

If $i=j$, then by Jensen's inequality and iteration, we have similarly
to before
\[
\E \bigl[ d(X_t, Y_t)^q \bigr] \leq\E \bigl[
\mathrm{e}^{-q\int
_0^t \rho
(I_s) \,\mathrm{d}s } \bigr] d(x,y)^q,
\]
where $q\in(0,1]$. By Lemma~\ref{lem:Per-Frob}, there exist $C,\eta
>0$ and $q\in(0,1]$ such that
\[
\E \bigl[ d(X_t, Y_t)^q \bigr] \leq C
\mathrm{e}^{-\eta t} d(x,y)^q.
\]
Now, for general $i$ and $j$, we have
\begin{eqnarray*}
\E \bigl[ \mathbf{d}(\mathbf{X}_t, \mathbf{Y}_t) \bigr]
& \leq&\E \bigl[ \sqrt{\one_{ T_c \geq t/2} \bigl(1 + V^q(X_t)
+ V^q(Y_t)\bigr)} \bigr]
\\
&&{}+ \E \bigl[ \sqrt{\one_{T_c \leq t/2} d(X_t,
Y_t)^q \bigl(1 + V^q(X_t) +
V^q(Y_t)\bigr)} \bigr],
\end{eqnarray*}
where $V(x)=d(x,x_0)$. Now, Cauchy--Schwarz inequality, Equation \eqref
{eq:coalescent}, Lemma~\ref{lem:Curv-Lyap} and Lemma~\ref
{lem:Lyap-cst} give
\begin{eqnarray*}
\E \bigl[ \sqrt{\one_{ T_c \geq t/2} \bigl(1 + V^q(X_t)
+ V^q(Y_t)\bigr)} \bigr] &\leq&\p (T_c \geq
t/2 )^{1/2} \E \bigl[1 + V^q(X_t) +
V^q(Y_t) \bigr]^{1/2}
\\
& \leq& C_c \mathrm{e}^{-\theta_c t/4} \bigl(1 + C_V
\mathrm{e}^{-
\lambda_V t} \bigl(V^{q} (x) + V^{q} (y) \bigr)
+ 2 K_V \bigr)^{1/2}.
\end{eqnarray*}
In the other hand, one has the bound
%
%
%e2.8 #&#
\begin{eqnarray}
\label{e:secondFact} &&\E \bigl[ \sqrt{\one_{T_c \leq t/2} d(X_t,
Y_t)^q \bigl(1 + V^q(X_t) +
V^q(Y_t)\bigr)} \bigr]
\nonumber
\\[-8pt]
\\[-8pt]
&&\quad\leq\E \bigl[ \one_{T_c \leq t/2} d(X_t,
Y_t)^q \bigr]^{1/2} \E \bigl[1 +
V^q(X_t) + V^q(Y_t)
\bigr]^{1/2}.
\nonumber
\end{eqnarray}
As a consequence of Lemmas~\ref{lem:Curv-Lyap} and~\ref
{lem:Lyap-cst}, we also have the bound
\begin{eqnarray*}
\E \bigl[ \one_{T_c \leq t/2} d(X_t, Y_t)^q
\bigr]^{1/2} &\leq& C \mathrm{e}^{- \eta t/2} \E \bigl[
d(X_{T_c}, Y_{T_c})^q \one_{T_c \leq
t/2}
\bigr]^{1/2}
\\
&\leq& C \mathrm{e}^{- \eta t/2} \E \bigl[ \bigl(V(X_{T_c})^q
+ V(Y_{T_c})^q \bigr)\one_{T_c \leq t/2}
\bigr]^{1/2}
\\
&\leq& C \mathrm{e}^{- \eta t/2} \bigl[C_V V^q(x_0)
+ C_V V^q(y_0) + 2K_V
\bigr]^{1/2}.
\end{eqnarray*}
Assembling these inequalities and using again Lemma~\ref{lem:Lyap-cst}
to bound the second factor in
\eref{e:secondFact}, we find that there exist constants $C>0$ and
$\lambda>0$ such that
\[
\E\bigl[\mathbf{d}(\mathbf{X}_t, \mathbf{Y}_t)\bigr]
\leq C\mathrm {e}^{-\lambda t} \bigl(1 +V(x) + V(y)\bigr),
\]
for every $t\geq0$ and $x,y\in E$. (Recall that $x$ and $y$ denote the
$E$-components of the initial conditions.) As a consequence of this
bound and the definition
of the Wasserstein distance, we deduce that
%
%
%e2.9 #&#
\begin{equation}
\label{eq:eqfinal} \mathcal{W}_\mathbf{d} (\boldsymbol{\mu}
\mathbf{P}_t, \boldsymbol{\nu} \mathbf{P}_t ) \leq C
\mathrm{e}^{-\lambda
t} \biggl(1 + \sum_{i\in F} \int
_E \bigl(V(x) \boldsymbol{\nu} (\mathrm{d}x,i) + V(x)
\boldsymbol{\mu}(\mathrm{d}x,i)\bigr) \biggr),
\end{equation}
for any two probability measures $\boldsymbol{\mu}$ and $\boldsymbol
{\nu}$. Now, mimicking the proof of \cite{HM11}, Corollary~4.10, we
can prove the existence of an invariant measure. More precisely, fix a
probability measure $\boldsymbol{\mu}$
and note that \eref{eq:eqfinal} implies that $(\boldsymbol{\mu}
\mathbf{P}_{n })_{n\geq0}$ is a Cauchy sequence with respect to the
distance $\mathcal{W}_\mathbf{d}$. We deduce that it converges to a
measure $\boldsymbol{\mu}_\infty$ verifying
\[
\boldsymbol{\mu}_\infty\mathbf{P}_{1}= \boldsymbol{
\mu}_\infty.
\]
It immediately follows that $\boldsymbol{\pi} = \int_0^1 \boldsymbol
{\mu}_\infty\mathbf{P}_{u} \,\mathrm{d}u$ is invariant, just like
in the
classical proof of the
Krylov--Bogolioubov criterion.
\end{pf*}

%s2.3 #&#
\subsection{Proof of Theorem \texorpdfstring{\protect\ref{th:dtv-cst}}{1.7}}

%Here we recall some definitions and theorems from \cite{HMS}. Usually
%the notion of Lyapunov function is associated to the notion of small
%set:

Before we start the proof proper,
we recall a version of Harris' theorem
(also called Foster, Lyapunov, Meyn-Tweedie, Doeblin in the literature) that
is suitable for our needs. This theorem yields exponential convergence
to stationarity for a process
which does not ``escape to infinity'' and verifies furthermore a
Doeblin-type condition.
More precisely, we use the following notion of a small set:
%

%de2.10 #&#
\begin{defi}
\label{def:small}
A set $A \subset X$ is small for the semigroup $(P_t)_{t\geq0}$ over a
Polish space $(X,d_X)$, if there exists a time $t > 0$
and a constant $\varepsilon>0$ such that
\[
d_{\TV} (\delta_x P_t,\delta_y
P_t) \leq1-\varepsilon
\]
for every $x, y \in A$.
\end{defi}

The classical Harris theorem \cite{HM11,MT93} then states that

%
%th2.11 #&#
\begin{theo}[(Harris)]
\label{th:Harris}
Let $(P_t)_{t\geq0}$ be a Markov semigroup over a Polish space
$(X,d_X)$ such that there exists a
Lyapunov function $V$ with the additional property that the sublevel
sets $\{x \in X | V(x) \leq C \}$ are small for
every $C>0$. Then $(P_t)_{t\geq0}$ has a unique invariant measure $\pi
$ and
\[
d_\TV(\delta_x P_t,\pi) \leq C
\mathrm{e}^{-\gamma_* t} \bigl(1+V(x)\bigr),
\]
for some positive constants $C$ and $\gamma_*$.
\end{theo}

Note that one does not really need that \textit{all} sublevel sets are
small and
one can have a slightly stronger conclusion by using a total variation
distance weighted by $V$, see, for example, \cite{HM11}, Theorem~1.3.

\begin{pf*}{Proof of Theorem~\ref{th:dtv-cst}}
By Lemma~\ref{lem:Lyap-cst}, $\mathbf{P}$ admits $V$ as Lyapunov
function so, by Harris' theorem,
it only remains
to show that $\{ V \leq C \}$ is small for $\mathbf{P}$, for every
$C>0$. Since $V$ is a Lyapunov function, there exists $t_*^{(1)}>0$ and
$K>K_V$ (with $K_V$ as in Lemma~\ref{lem:Lyap-cst}) such that
\[
\forall t\geq t_*^{(1)},\qquad\E \bigl[V(X_t) \bigr] \le K,
\]
uniformly over all $x\in E$ such that $V(x)\leq C$. Therefore, if
$\mathbf{X}$ is a process generated by $\mathbf{L}$, it follows from
Markov's inequality that
\[
\p \bigl(V(X_{t}) \le2K \bigr) \ge\tfrac{1}{2},
\]
uniformly over $t \ge t_*^{(1)}$.

Let now $i_0 \in F$ be as in the statement. Since $A=\{ V \leq2K \}$
is small for $P^{(i_0)}$, we obtain some $t_0>0$ and $\varepsilon>0$,
such that for all $x,y\in A$ there
exists a coupling $(Z^{i_0,x}_t,Z^{i_0,y}_t)$ verifying
%
%
%e2.10 #&#
\begin{equation}
\label{e:boundZ} \p \bigl( Z^{i_0,x}_t = Z^{i_0,y}_t
\bigr) \geq\varepsilon,\qquad t \ge t_0,
\end{equation}
and $Z^{i_0,x}_t$, $Z^{i_0,y}_t$ have respective law $\delta_x
P^{(i_0)}_t$, $\delta_y P^{(i_0)}_t$.

By the irreducibility of the process $I$, one can find $t_* >
t_*^{(1)}$ and $\delta> 0$ such that $\p(I_s = i_0, \forall s \in
[t_*,t_*+ t_0 ]) > \delta$, uniformly over the starting distributions.
Let now
$(\mathbf{X}_t,\mathbf{Y}_t)$ be the following coupling:
\begin{itemize}
\item the Markov chains $I$ and $J$ are independent over $t\in[0, t_*
+ t_0]$;
\item the processes $X$ and $Y$ are independent over $t\in[0, t_*]$;
\item conditionally on the set
\[
B=\bigl\{V(X_{t_*}) \le2K, V(Y_{t_*}) \le2K,
I_s = J_s = i_0, \forall s \in[t_*,t_*+
t_0 ] \bigr\},
\]
the processes $X$ and $Y$ are coupled in such a way to verify \eqref
{e:boundZ}, over $t\in[t_*, t_* + t_0]$;
\item conditionally on $B^c$, they are coupled independently from each other.
\end{itemize}
The Markov property gives
%
%
%e2.11 #&#
\begin{equation}
\p\bigl(V(X_{t_*}) \le2K, I_s = i_0, \forall
s \in[t_*,t_*+ t_0 ]\bigr) \ge \frac{\delta}{2},
\end{equation}
and so $\p(B)\ge\delta^2 /4$. Combining this inequality with
\eqref{e:boundZ}, we conclude that
$\p(\mathbf{X}_{t_*+t_0} = \mathbf{Y}_{t_*+t_0}) \ge\delta^2
\varepsilon/4$, uniformly over all
initial conditions $\mathbf{x}$ and $\mathbf{y}$ with $V(x) \le C$
and $V(y)\le C$, as required.
\end{pf*}

%s3 #&#
\section{Non-constant jump rates}
\label{sect:non-cst}

In all of this section, we now assume that $a$ depends non-trivially on
its first component,
so that $I$ by itself is not a Markov process anymore. We want to use again
Lemma~\ref{lem:Per-Frob} to show that $\mathbf{X}$ converges, but
this time we cannot use it directly
on $I$. The idea is to consider an auxiliary process which does not
depend to $X$ and which will bound
$(\rho(I_t))_{t\geq0}$ or $(\lambda(I_t))_{t\geq0}$. More
precisely, we will assume
the following assumption.
%

%as3.1 #&#
\begin{hypo}[(Birth--death type criterion in the non constant case)]
\label{hyp-iborne}
There exist $\bar{n} \in\N$ and a partition $(F_n)_{0 \leq n \leq
\bar{n}}$ of $F$ such that
\[
\forall n \leq\bar{n}, \forall i \in F_n, \forall j \notin
F_{n-1} \cup F_n \cup F_{n+1}, \forall x\in E,
\qquad a(x,i,j)=0,
\]
where we have set $F_{-1} =F_{\bar{n} +1} = \emptyset$.
Let $(L_t)_{t\geq0}$ be the continuous-time Markov chain on $\{
0,\ldots,\bar n\}$
with generator
%
%
%e3.1 #&#
\begin{equation}
\label{eq:def-bd-L} G f(n) = b(n) \bigl( f(n+1) - f(n) \bigr) + d(n) \bigl( f(n-1) -
f(n) \bigr),
\end{equation}
for every $n \leq\bar{n}$, where $d(0)=b(\bar{n})=0$,
\[
b(n) = \inf_{x\in E} \inf_{i \in F_n} \sum
_{j \in F_{n+1}} a(x,i,j) > 0,
\]
for $n < \bar n$ and
\[
d(n) = \sup_{x\in E} \sup_{i \in F_n} \sum
_{j \in F_{n-1}} a(x,i,j) > 0,
\]
for $n > 0$.
\end{hypo}

%
%re3.2 #&#
\begin{Rq}
The process with generator $G$ is irreducible, non-explosive and
positive recurrent.
We will henceforth denote its invariant measure by $\nu$.
\end{Rq}

If Assumption~\ref{hyp-iborne} holds then, for every $i\in F$, we
denote by $n_i$ the only $n\leq\bar{n}$ verifying $i\in F_n$.
Let us recall that, for every $n\leq\bar{n}$, the invariant measure
$\nu$ is given by
\[
\nu(n)= \nu(0) \prod_{k=1}^n
\frac{b(k-1)}{d(k)} \quad\mbox {and}\quad\nu (0) = (1 + \Xi )^{-1},
\]
where
\[
\Xi= \sum_{n=1}^{\bar{n}} \frac{b(0) \cdots b(n-1)}{d(1) \cdots d(n)}.
\]
Now we can state two slight generalisations of Theorems~\ref
{th:W-noncst} and~\ref{th:dtv-noncst}. The first one is
%

%th3.3 #&#
\begin{theo}[(Wasserstein exponential ergodicity)]
\label{th:W-noncst2}
Suppose that Assumptions~\ref{hyp:a-reg}, \ref{hyp:rec}, \ref
{hypo:Curv}, and~\ref{hyp-iborne} hold. If
\[
\sum_{n = 0}^{\bar{n}} \nu(n) \alpha(n) > 0,
\]
where $(\alpha(n))_{n\geq0}$ is an increasing sequence verifying
$\alpha(n)\leq\inf_{i \in F_n} \rho(i)$,
then there exist a probability measure $\boldsymbol{\pi} $ and some
constants $C, \lambda, t_0>0$
and $q\in(0,1]$ such that
\[
\forall t\geq t_0,\qquad\mathcal{W}_{\mathbf{d}} (\delta
_{\mathbf
{y}_0} \mathbf{P}_t, \boldsymbol{\pi} ) \leq C \mathrm
{e}^{-\lambda t} \biggl(1+\sum_{i\in F} \int
_E d(y_0,x)^q \boldsymbol{\pi} (
\mathrm{d}x,i) \biggr),
\]
for every $\mathbf{y}_0 = (y_0,j_0) \in\mathbf{E}$. Here, the
distance $\mathbf{d}$ on $\mathbf{E}$ was defined in \eqref{eq:dgras}.
\end{theo}

If Assumption~\ref{hyp-iborne} holds with $\bar{n}=0$ then all
contraction parameters are positive and we recover \cite{MLMZ-Quant}, Theorem~1.15. If it holds with $\bar{n}=1$, then we have the
on--off criterion which was given in introduction. We can also state the
analogous result in the setting
of Theorem~\ref{th:dtv-noncst}:

%
%th3.4 #&#
\begin{theo}[(Exponential ergodicity)]
\label{th:dtv-noncst2}
Suppose that Assumptions~\ref{hyp:a-reg}, \ref{hyp:rec}, \ref
{hypo:Curv} and~\ref{hyp-iborne} hold and there exist $i_0\in F$ and
$t_0 \geq0$ such that the sublevel sets of $V$ are small for
$P^{(i_0)}_t$, for every $t\geq t_0$. If
\[
\sum_{n = 0}^{\bar{n}} \nu(n) \alpha(n) > 0,
\]
where $(\alpha(n))_{n\geq0}$ is an increasing sequence verifying
$\alpha(n)\leq\inf_{i \in F_n} \lambda(i)$,
then there exist a probability measure $\boldsymbol{\pi} $ and two
constants $C, \lambda>0$ such that
\[
\forall t \geq0,\qquad d_\TV (\delta_{\mathbf{x}}
\mathbf{P}_t, \boldsymbol{\pi} ) \leq C \mathrm{e}^{-\lambda t}
\bigl(1+V(x)\bigr)
\]
for every $\mathbf{x}=(x,i) \in\mathbf{E}$.\vadjust{\goodbreak}
\end{theo}

We do not give the proofs of Theorem~\ref{th:dtv-noncst} and Theorem~\ref{th:dtv-noncst2},
as their proofs are very similar to the proof of Theorem~\ref
{th:dtv-cst}, combined
with the argument of Lemma~\ref{lem:Lyap-noncst} below.
To prove Theorem~\ref{th:W-noncst2} however, we cannot use classical
Harris' Theorem.
Its proof follows the same idea as the proof of Theorem~\ref
{th:W-cst}, but there is
no direct equivalent to the meeting time. Instead, we use a weak
version of Harris' Theorem which
yields geometric ergodicity under the existence of a Lyapunov function
and a modified
``small set'' condition. This theorem was previously applied to the
stochastic Navier--Stokes
equation \cite{HM09}, stochastic delay differential equations \cite
{HMS}, and linear
response theory \cite{HM10}. It is an extension of the classic Harris'
Theorem which
allows to deal with some degenerate
examples like the one given in \eqref{eq:exemple-intro}.

%s3.1 #&#
\subsection{Weak form of Harris' Theorem}

As already mentioned earlier, there are situations in which we cannot expect
convergence in total variation. The problem here is that bounded sets
may not be small sets.
We will therefore replace the notion of small set by the following
notion of ``closedness'' between
transition probabilities introduced in \cite{HMS}, which takes into
account the topology of the
underlying space $X$.
%

%de3.5 #&#
\begin{defi}[($d$-small set)]
Let $P$ be a Markov operator over a Polish space $X$ endowed with a
distance $d_X \dvtx X \times X \mapsto[0,1]$. A set $A\subset X$ is
said to be $d_X$-small if there exists a constant $\varepsilon$ such that
\[
\mathcal{W}_{d_X} (\delta_x P,\delta_y P)
\leq1-\varepsilon,
\]
for every $x, y \in A$.
\end{defi}

This notion is a generalisation of the notion of small set, since small
sets are $d$-small for the trivial distance.
This definition can also be extended to situations when $d$ is not a
distance \cite{HMS}.
As remarked in that paper,
having a Lyapunov function $V$ with $d$-small sublevel sets cannot be
sufficient to imply the ergodicity of a Markov semigroup. To obtain
some convergence result, we further impose that $d$
is contracting for our semigroup:
%

%de3.6 #&#
\begin{defi}[($d$-contracting operator)]
Let $P$ be a Markov operator over a Polish space $X$ endowed with a
distance $d_X\dvtx X \times X \mapsto[0,1]$. The distance $d_X$ is
said to be contracting for $P$ if there exists $\alpha< 1$ such that
the bound
\[
\mathcal{W}_{d_X}(\delta_x P, \delta_y P)
\leq\alpha d_X (x, y)
\]
holds for every $x, y \in X$ verifying $d(x, y) < 1$.
\end{defi}

Note that this condition alone is not sufficient to guarantee the
convergence of
transition probabilities toward a unique invariant measure since we
only impose a contraction
when $d(x, y) < 1$. In typical situations, ``most'' pairs $(x,y)$ may
satisfy $d(x,y) = 1$,
as would be the case for the total variation distance.
However, when combined with the existence of a Lyapunov function $V$
that has $d$-small sublevel sets, it gives geometrical ergodicity
(\cite{HMS}, Theorem~4.7):

%
%th3.7 #&#
\begin{theo}[(Weak form of Harris' Theorem)]
\label{th:Weak-Harris}
Let $(P_t)_{t\geq0}$ be a Markov semigroup over a Polish space $X$
admitting a continuous Lyapunov
function $V$. Assume furthermore that there exist $t^*> t_* > 0$ and a
distance $d_X \dvtx X \times X \mapsto[0,1]$
which is contracting for $P_t$ and such that the sublevel set $\{ x \in
X | V(x) \leq4 K_V \}$ is $d_X$-small for $P_t$, for every $t\in
[t_*,t^*]$. Here $K_V$ is as in Definition~\ref{def:lyap}. Then,
$(P_t)_{t\geq0}$ has an invariant probability measure $\pi$.
Furthermore, defining
\[
\delta_X (x, y) = \sqrt{ d_X(x, y) \bigl(1 + V (x) + V
(y)\bigr)},
\]
there exist $r>0$ and $t_0 > 0$ such that
\[
\forall t \geq t_0,\qquad\mathcal{W}_{\delta_X}(\mu
P_t, \nu P_t) \leq \mathrm{e}^{- rt}
\mathcal{W}_{\delta_X}(\mu, \nu),
\]
for all of probability measures $\mu, \nu$ on $X$.
\end{theo}

%
%re3.8 #&#
\begin{Rq}[(On the contracting distances)]
The main difficulty when applying the previous theorem is to find a
contracting distance. The construction of
this distance represents the main part of our paper. In \cite{HM10},
there is a general way to build a
contracting distance of a Markov operator $P$ over a Banach space
$(\mathbb{B}, \Vert\cdot\Vert)$,
based on a gradient estimate for $P$ and the existence of a
super-Lyapunov function.
This technique was efficient in \cite{HM10,HM09}.
\end{Rq}

%s3.2 #&#
\subsection{Construction of a Lyapunov function}

As in the constant case, we first show that if each underlying Markov process
verifies a weaker form of the drift condition \eqref{eq:CD} then
$\mathbf{X}$
possesses a Lyapunov function:

%
%le3.9 #&#
\begin{lem}[(Construction of a Lyapunov function)]
\label{lem:Lyap-noncst}
Suppose that Assumptions~\ref{hyp:a-reg}, \ref{hyp:rec}, \ref
{hyp:Lyap} and~\ref{hyp-iborne} hold, if
\[
\sum_{n\geq0} \nu(n) \alpha(n) > 0,
\]
where $(\alpha(n))_{n\geq0}$ is an increasing sequence verifying
$\alpha(n)\leq\inf_{i \in F_n} \lambda(i)$,
then there exist $C_V, K_V, \lambda_V > 0$ and $q\in(0,1)$ such that,
for all $t \ge0$ and all $(x,i) \in\mathbf{E}$, the bound
%
%
%e3.2 #&#
\begin{equation}
\label{e:wantedLyap} \mathbf{P}_t V^{q} (x,i) \leq
C_V \mathrm{e}^{- \lambda_V t} V^{q} (x) +
K_V
\end{equation}
holds.
\end{lem}
\begin{pf}
Recall again that Jensen's inequality gives this weaker form of \eqref
{eq:Lyap-underlying}:
\[
\bigl(P^{(i)}_t V^{q} \bigr) (x) \leq
\mathrm{e}^{- q \alpha(i) t} V^{q}(x) + K^{q},
\]
for every $x\in E$ and $q\in(0,1]$.
Note also that, as a consequence of the Markov property,
\eref{e:wantedLyap} follows if we are able to find \textit{some}
$T > 0$ and constants $C, K>0$ and $q \in(0,1]$ such that
%
%
%e3.3 #&#
\begin{equation}
\label{e:wantedLyap2} \mathbf{P}_T V^{q} (x,i) \leq
\tfrac{1}{2} V^{q} (x) + K,
\end{equation}
and such that\vspace*{1pt}
%
%
%e3.4 #&#
\begin{equation}
\label{e:wantedLyap3} \mathbf{P}_t V^{q} (x,i) \leq C
V^{q} (x) + K,
\end{equation}
for all $t \in[0,T]$. In order to find such a time $T$, we will build
a process which couples a copy of $\mathbf{X}$ with the birth and death
process $L$ of Assumption~\ref{hyp-iborne}.
We define a generator $\mathcal{G}$ on $\mathbf{E}\times\{0,\ldots
,\bar n\}$ by\vspace*{1pt}
\begin{eqnarray*}
\mathcal{G} f(x,i,l) &=& \mathcal{L}^{(i)} f(x,i,l) + \sum
_{j \in F} a(x,i,j) \bigl(f(x,j,l)-f(x,i,l) \bigr)
\\[1pt]
&&{} + b(l) \bigl(f(x,i,l+1)-f(x,i,l) \bigr) + d(l) \bigl(f(x,i,l-1)-f(x,i,l)
\bigr)
\end{eqnarray*}
for $l\neq n_i$. For $l = n_i$ on the other hand, we set\vspace*{1pt}
\begin{eqnarray*}
\mathcal{G} f(x,i,l) &=& \mathcal{L}^{(i)} f(x,i,l) + \sum
_{j \in F_{l -1}} a(x,i,j) \bigl(f(x,j,l-1)-f(x,j,l) \bigr)
\\[1pt]
&&{} + \biggl(d(l) - \sum_{j \in F_{l -1}} a(x,i,j) \biggr)
\bigl(f(x,i,l-1)-f(x,j,l) \bigr)
\\[1pt]
&&{} + \sum_{j \in F_{l}} a(x,i,j) \bigl( f(x,j,l) -
f(x,i,l) \bigr)
\\[1pt]
&&{} + \frac{b(l)}{\sum_{k \in F_{l +1}} a(x,i,k)} \sum_{j \in F_{l
+1}} a(x,i,j)
\bigl( f(x,j,l+1) -f(x,j,l) \bigr)
\\[1pt]
&&{} + \frac{\sum_{k \in F_{l +1}} a(x,i,k) - b(l)}{\sum_{k \in F_{l
+1}} a(x,i,k)} \sum_{j \in F_{l +1}} a(x,i,j)
\bigl(f(x,j,l)-f(x,j,l) \bigr).
\end{eqnarray*}
In words, as long as $L\neq n_I$, $L$ and $\mathbf{X}$ move
independently from each other until the time where $n_I$ and $L$ agree.
After that time,
the coupling is designed in such a way that one always has $n_I \ge L$.
If we start the process with an initial condition $(x,i,l)$ such that
$n_i \ge l$, this construction ensures in particular that,
for all times, one has\vspace*{1pt}
\[
\alpha(I_t) \geq\alpha(L_t).
\]
We now denote by $\{\tau_n\}_{n \ge1}$ the times at which the process
$I_t$ jumps
and by $N_t$ the number of such jumps before time $t$.

With these notations at hand, we then have\vspace*{1pt}
%
%
%e3.5 #&#
\begin{eqnarray}
\label{e:goodBound} \mathbf{P}_t V^{q} (x) &=& \E \bigl[
P^{(I_{\tau_{N_t}})}_{t -\tau_{N_t}} V^{q} (X_{\tau
_{N_t}}) \bigr] \leq
\E \bigl[ \mathrm{e}^{-\alpha(I_{\tau_{N_t}}) (t -\tau
_{N_t})} V^{q} (X_{\tau_{N_t}}) +
K^q \bigr]
\nonumber
\\
&\le&\E \bigl[ \mathrm{e}^{-\int_{\tau_{N_t}}^t \alpha(L_{s}) \,
\mathrm{d}s} V^{q} (X_{\tau_{N_t}})
\bigr] + K^q
\\
&=& \E \bigl[\mathrm{e}^{-\int_{\tau_{N_t}}^t \alpha(L_{s}) \,
\mathrm{d}s} P^{(I_{\tau
_{N_t-1}})}_{\tau_{N_t} -\tau_{N_t-1}}
V^{q} (X_{\tau
_{N_t-1}}) \bigr] + K^q
\nonumber
\\
&\leq&\cdots\le\E \bigl[ \mathrm{e}^{-\int_0^t q \alpha(L_s) \,
\mathrm{d}s } \bigr] V^q(x) +
K^q \E \biggl[ \sum_{n \le N_t}
\mathrm{e}^{-q \int_{\tau_n}^t
\alpha(L_s) \,\mathrm{d}s} \biggr].
\nonumber
\end{eqnarray}
Now, using Lemma~\ref{lem:Per-Frob}, there exist $C,\eta>0$ and $q\in
(0,1]$ such that
%
%
%e3.6 #&#
\begin{equation}
\label{eq:lambda-eta} \E \bigl[ \mathrm{e}^{-\int_0^t q \alpha(L_s) \,\mathrm{d}s
} \bigr] \leq C
\mathrm{e}^{-\eta t}.
\end{equation}
Hence, in view of \eref{e:wantedLyap2} and \eref{e:wantedLyap3},
it only remains to prove that, for any fixed time $T$, one has the bound
\[
\sup_{t \le T} \E \biggl[ \sum_{n \le N_t}
\mathrm{e}^{-q \int
_{\tau_n}^t
\alpha(L_s) \,\mathrm{d}s} \biggr] < + \infty.
\]
Since the function $\alpha$ is bounded from below and the function $t
\mapsto N_t$ is
increasing, this boils down to the bound $\E N_T < \infty$, which is a simple
consequence of the fact that by Assumption~\ref{hyp:a-reg}
the jump rates are also bounded from above.
\end{pf}

%s3.3 #&#
\subsection{The contracting distance}

This section is divided in three parts. We introduce the distance
$\widetilde{d}$ that we will use in Theorem~\ref{th:Weak-Harris}, we
build our coupling in such a way that $\widetilde{d}$ will
be contracting for it, and we finally prove that it is indeed contracting.

%s3.3.1 #&#
\subsubsection{Definition of \texorpdfstring{$\widetilde{d}$}{widetilded}}

Here, we build a distance $\widetilde{d}\dvtx (E\times F) \times
(E\times F) \rightarrow[0,1]$ such that there exist $t_*>0$ and
$\alpha\in(0,1)$ verifying
%
%
%e3.7 #&#
\begin{equation}
\label{eq:contraction} \widetilde{d}(\mathbf{x},\mathbf{y})<1 \quad\Rightarrow\quad
\forall t\geq t_*, \qquad\mathcal{W}_{\tilde{d}}(\delta_\mathbf{x}
P_t, \delta _\mathbf {y} P_t) \leq\alpha
\widetilde{d}(\mathbf{x},\mathbf{y}).
\end{equation}
where $\mathbf{x}=(x,i)$ and $\mathbf{y}=(y,j)$ belong to $E\times
F$. Since we can say nothing when $i\neq j$, we will take $\widetilde
{d}(\mathbf{x},\mathbf{y})$ constant equal to $1$ in this case. When
$i=j$ we want to use Assumption~\ref{hypo:Curv} to prove a decay. But
it is more useful to ``decrease the contraction'' of the underlying
Markov semigroup. More precisely, by Jensen inequality, Assumption~\ref
{hypo:Curv} gives
\[
\mathcal{W}_{d^q}\bigl(\mu P^{(i)}_t, \nu
P^{(i)}_t\bigr) \leq\mathrm{e}^{-q
\rho(i) t}
\mathcal{W}_{d^q}(\mu,\nu),
\]
for all $t\geq0 $, $q\in(0,1]$ and every probability measures $\mu
,\nu$. Finally, we define $\widetilde{d}$ by
\[
\widetilde{d}(\mathbf{x},\mathbf{y}) = \one_{i\neq j} +
\one_{i=j} \bigl( \delta^{-1} d^q(x,y) \wedge1
\bigr),
\]
where $\delta>0$ will be determined later. Now, if a realisation of
the coupling $(\mathbf{X}_t,\mathbf{Y}_t)_{t\geq0} =
((X_t,I_t),(Y_t, J_t))_{t\geq0}$ starting from $(\mathbf{x},\mathbf
{y})$, verifies $\widetilde{d}(\mathbf{x},\mathbf{y})<1$, then
$I_0=J_0=i=j$. So, we will try to build our coupling in
such a way that $I$ and $J$ remain equal for as long as possible. More
precisely, if we set
%
%
%e3.8 #&#
\begin{equation}
\label{eq:def-T} T = \inf\{ s \geq0 | I_s \neq J_s\},
\end{equation}
then we will prove that there exists $K>0$ and a choice of coupling
such that
\[
\p(T < \infty) \leq K d(x,y).
\]

%s3.3.2 #&#
\subsubsection{Construction of our coupling}

Here, we fix $\mathbf{x}=(x,i)$, $\mathbf{y}=(y,j)$ in $\mathbf{E}$
and we let $t>0$. Let $r\geq0$ and $(N_t)_{t\geq0}$ be a Poisson
process of intensity $r$ with $N_t = \sum_{n\geq0} \one_{\{\tau_n
\leq
t\}}$
and $\tau_n= \sum_{k=1}^n E_k$ for a family $(E_k)_{k \geq0}$ of
i.i.d. exponential variables and
$\tau_0=0$. We assume that $r\geq2 \bar{a}$, that
is $r$ is larger than the jump rates of $I$ or $J$. As in the proof of
Theorem~\ref{th:W-cst}, we give the construction of our coupling
$(\mathbf{X}, \mathbf{Y})$ at the jump
times of $N$. Let $n\in\{ 0,\ldots, N_t\}$, we consider the following
dynamics:
\begin{itemize}
\item If $I_{\tau_n} \neq J_{\tau_n}$, then $X_{s}$ and $Y_{s}$
evolve independently for every $s\in[\tau_n, \tau_{n+1} \wedge t)$.
\item If $I_{\tau_n} = J_{\tau_n}$, then by Assumption~\ref
{hypo:Curv}, we can couple $X$ and $Y$ in such a way that
\[
\E \bigl[ d(X_{\tau_{n+1} \wedge t},Y_{\tau_{n+1} \wedge t}) | \mathcal{G}_{\tau_n}
\bigr] \leq\mathrm{e}^{-\rho(I_{\tau
_n})(\tau
_{n+1}\wedge t - \tau_n)} d(X_{\tau_n},Y_{\tau_n}),
\]
where $\mathcal{G}_n =\sigma\{ (\mathbf{X}_{\tau_n},\mathbf
{Y}_{\tau_n}), (\tau_k)_{k\geq0} \}$.
\end{itemize}
At the jump times of $N$ the situation is different since $I$ or $J$
may jump. We will optimise the chance that $I$ and $J$ jump
simultaneously. For each $n\in\N^*$, we cut $[0,1]$ in four parts
$I^n_0, I^n_1, I^n_2, I^n_3$ in such a way that
\begin{eqnarray*}
\lambda\bigl(I^n_0\bigr)&=& \frac{1}{r} \sum
_{j \in F} \bigl( a(X_{\tau_n -}, I_{\tau_n},
j) - a (Y_{\tau_n -}, I_{\tau_n}, j) \bigr)_+,
\\
\lambda\bigl(I^n_1\bigr)&=& \frac{1}{r} \sum
_{j \in F} \bigl( a(Y_{\tau_n -}, I_{\tau_n},
j) - a (X_{\tau_n -}, I_{\tau_n}, j) \bigr)_+,
\\
\lambda\bigl(I^n_2\bigr)&=& \frac{1}{r} \sum
_{j\in F} a(X_{\tau_n -}, I_{\tau
_n}, j)
\wedge a(Y_{\tau_n -}, I_{\tau_n}, j),
\\
\lambda\bigl(I^n_3\bigr)&=& 1 - \frac{1}{r} \sum
_{j\in F} a(X_{\tau_n -}, I_{\tau_n}, j)
\vee\sum_{j\in F} a(Y_{\tau_n -},
I_{\tau_n}, j),
\end{eqnarray*}
where $\lambda$ is the Lebesgue measure and $(x)_+= \max(x,0)$. Let
$(U_n)_{n\geq0}$ be a sequence of i.i.d. random variables uniformly
distributed on $[0,1]$, we couple $I$ and $J$ at the jump times as follows:
\begin{itemize}
\item For $U_n \in I^n_0$, $I$ jumps, but $J$ does not jump.
\item For $U_n \in I^n_1$, $J$ jumps, but $I$ does not jump.
\item For $U_n \in I^n_2$, $I$ and $J$ both jump simultaneously to the
same location.
\item For $U_n \in I^n_3$, $I$ and $J$ both stay in place.
\end{itemize}
The second components, $X$ and $Y$, do not jump. Finally, we also
couple $\mathbf{X}$ and $\mathbf{Y}$ with a continuous Markov chain
$L$ which only depend to $U$ and $N$ and which verifies
\[
\forall t\geq0,\qquad\rho(I_t) \geq\alpha(L_t).
\]
This Markov chain $L$ is constructed as in the proof of Lemma~\ref
{lem:Lyap-noncst}.

%
%re3.10 #&#
\begin{Rq}
This coupling is not quite Markovian since, between times $\tau_n$ and
$\tau_{n+1}$, it already
uses information about the pair $(X_t, Y_t)$ at time $\tau_{n+1}$.
However, in many situations to which our results apply there exists a
Markovian coupling
with generator $\mathbb{L}^{(i)}$
which yields a good coupling for each of the underlying processes.
In this case, we can make our coupling Markovian with generator
\begin{eqnarray*}
\mathbb{L} f(\mathbf{x},\mathbf{y},n) &=& \mathbb{L}^{(i)} f(
\mathbf{x},\mathbf{y},n) + \sum_{k \in F} \bigl( a(x,i,k)
- a (y,j,k) \bigr)_+ f\bigl((x,k),\mathbf{y},n+1\bigr)
\\
&&{} + \sum_{k \in F} \bigl( a(y,j,k) - a (x,i,k)
\bigr)_+ f\bigl(\mathbf {x},(y,k),n+1\bigr)
\\
&&{} + \sum_{k \in F} a(x,i,k) \wedge a(y,j,k) f
\bigl((x,k),(y,k),n+1\bigr)
\\
&&{} + \biggl(r - \sum_{k \in F} a(x,i,k) \vee a(y,j,k)
\biggr)f(\mathbf {x},\mathbf{y},n+1) - r f(\mathbf{x},\mathbf{y},n).
\end{eqnarray*}
\end{Rq}

%s3.3.3 #&#
\subsubsection{The distance \texorpdfstring{$\widetilde{d}$}{widetilded} is contracting for
\texorpdfstring{$\mathbf{P}$}{P}}

In this subsection, we show that the distance $\widetilde d$ defined above
is indeed contracting for the coupling constructed in the previous subsection.
This is formulated in the following result.

%
%le3.11 #&#
\begin{lem}
\label{lem:d-contract}
Let $(\mathbf{X}_t, \mathbf{Y}_t)_{t\geq0}$ be the coupling of the
previous section. Under the assumptions of Theorem~\ref{th:W-noncst2},
we can choose $r$ and $\delta$ in such a way that
\[
\forall t\geq t_*, \qquad\E \bigl[ \widetilde{d} (\mathbf{X}_t,
\mathbf {Y}_t) \bigr] \leq\gamma\widetilde{d}(\mathbf{x},\mathbf{y}),
\]
for some $\gamma\in(0,1)$ and $t_*>0$, and all $\mathbf{x},\mathbf
{y} \in E\times F$ verifying $\widetilde d(\mathbf{x},\mathbf{y})<1$.
\end{lem}

\begin{pf}
Recall that since $\widetilde d(\mathbf{x},\mathbf{y})<1$ one has
$I_0=J_0$ and that $T$, defined in \eqref{eq:def-T}, denotes the first
time of separation of $I$ and $J$. Using Lemma~\ref{lem:Per-Frob},
there exist $q\in(0,1]$ and $C, \eta>0$ such that
\begin{eqnarray*}
\E \bigl[ \widetilde{d}(\mathbf{X}_t,\mathbf{Y}_t)
\bigr] &\leq&\E \biggl[\one_{\{T= \infty\}}\frac{1}{\delta} d^q(X_t,Y_t)
+ \one _{\{T<+ \infty\}} \biggr]
\\
&\leq&\frac{1}{\delta} \E \bigl[\mathrm{e}^{-\int_0^t q\alpha
(L_s) \,\mathrm{d}s} \bigr]\E
\bigl[d^q(x,y)\bigr] + \p (T<+ \infty )
\\
&\leq& C \mathrm{e}^{- \eta t} \widetilde{d}(\mathbf{x},\mathbf {y}) + \p
(T<+ \infty ).
\end{eqnarray*}
Here, we have used the fact that
\begin{eqnarray*}
\E \bigl[\one_{\{T= \infty\}} d^q(X_t,Y_t)
\bigr] &\leq&\E \bigl[\one_{\{T\geq\tau_{N_t}\}} \mathrm{e}^{- q\alpha(
L_{\tau
_{N_t}}) (t- \tau_{N_t})}
d^q(X_{\tau_{N_t}},Y_{\tau_{N_t}}) \bigr]
\\
&\leq&\E \bigl[\one_{\{T\geq\tau_{N_t}\}} \mathrm{e}^{- q\alpha(
L_{\tau
_{N_t}}) (t- \tau_{N_t})} \E
\bigl[d^q(X_{\tau_{N_t}},Y_{\tau
_{N_t}}) | \mathcal{G}_n
\bigr] \bigr]
\\
&\leq&\E \bigl[\one_{\{T\geq\tau_{N_t-1}\}} \mathrm{e}^{- \int
_{\tau
_{N_t-1}}^t q \alpha( L_{s}) \,\mathrm{d}s}
d^q(X_{\tau
_{N_t-1}},Y_{\tau_{N_t
-1}}) \bigr]
\\
&\leq&\E \bigl[\mathrm{e}^{-\int_0^t q \alpha( L_{s}) \,\mathrm
{d}s} \bigr]\E \bigl[d^q(x,y)
\bigr].
\end{eqnarray*}
It remains to obtain a bound on $\p (T<+ \infty )$. Since
$I$ and $J$ can only
jump when $N$ jumps, $T$ can be finite only if it is one of the jump
times of $N$. So, we set
\[
A_n = \{T = \tau_n\} = \{ T\geq\tau_{n}
\mbox{ and } I_{\tau_n} \neq J_{\tau_n}\}.
\]
By Assumption~\ref{hyp:a-reg}, we have
\begin{eqnarray*}
\p ( A_n ) &=& \p \bigl(\bigl\{ U_n \in
I^n_0 \cup I^n_1 \cup
I^n_3\bigr\} \cap\{ T\geq \tau_{n} \} \bigr)
\\
&\leq&\E \biggl[\frac{2 \one_{\{T \geq\tau_n\}} \sum_{j\in F}
|a(X_{\tau_n -},I_{\tau_n -},j) - a(Y_{\tau_n -},I_{\tau_n
-},j)|}{r} \biggr]
\\
&\leq&\E \biggl[ \biggl(\frac{2 \one_{\{T \geq\tau_n\}} \sum_{j\in F}
|a(X_{\tau_n -},I_{\tau_n -},j) - a(Y_{\tau_n -},I_{\tau_n
-},j)|}{r} \biggr)^q \biggr]
\\
&\leq&\frac{2^q \kappa^q}{r^q} \E \bigl[ d(X_{\tau_n -},Y_{\tau_n
-}
)^q \bigr] \leq\frac{2^q \kappa^q}{r^q} \E \bigl[\mathrm{e}^{-q \int_0^{\tau_n}
\alpha(L_s) \,\mathrm{d}s}
\bigr] d(x,y)^q.
\end{eqnarray*}
Hence,
\[
\p ( T < \infty ) = \sum_{n\geq1} \p ( A_n
) \leq\frac{2^q \kappa^q}{r^q} d(x,y)^q \sum_{n\geq1}
\E \bigl[\mathrm{e}^{-q \int_0^{\tau_n} \alpha(L_s) \,\mathrm{d}s} \bigr].
\]
Now, similarly to the proof of Lemma~\ref{lem:Lyap-noncst}, provided
that $r$ is sufficiently large, there exist $C'>0$ and $\varepsilon>0$
verifying
\[
\sum_{n\geq1} \E \bigl[\mathrm{e}^{-q \int_0^{\tau_n} \alpha(L_s)
\,\mathrm{d}s}
\bigr] \leq\sum_{n\geq1} C' \mathrm
{e}^{-\varepsilon n} =: \tilde C < +\infty.
\]
Combining these bounds, we obtain the estimate
\[
\E \bigl[ \widetilde{d}(\mathbf{X}_t,\mathbf{Y}_t)
\bigr] \leq \biggl(C \mathrm{e}^{- \eta t} + \frac{(2 \kappa)^q \tilde
C}{r^q} \delta \biggr)
\widetilde{d}(\mathbf{x},\mathbf{y}).
\]
First making $\delta$ sufficiently small and then taking $t$
large enough, we thus obtain the announced result.
\end{pf}

%s3.4 #&#
\subsection{Bounded sets are \texorpdfstring{$\widetilde{d}$}{widetilded}-small}

%?

Here, we prove that if a set is bounded then it is $\widetilde{d}$-small.

%
%le3.12 #&#
\begin{lem}
\label{lem:d-small}
Under the assumptions of Theorem~\ref{th:W-noncst2}, if $S \subset
E\times F $ is of bounded
diameter in the sense that
\[
R=\sup\bigl\{ d(x,y) | \mathbf{x}, \mathbf{y} \in S \bigr\} < + \infty,
\]
then there exist $t_*, t^*>0$ such that $S$ is $\widetilde{d}$-small
for $P_t$, for all $t\in[t_*,t^*]$.
\end{lem}

\begin{pf}
Let $\mathbf{x}=(x,i)$ and $\mathbf{y}=(y,j)$ be two different points
of $S$. By
the assumptions of Theorem~\ref{th:W-noncst2}, there exists $i_0 \in
F$ such that $\rho(i_0)>0$.
Let $(\mathbf{X}_t)_{t\geq0}$ and $(\mathbf{Y}_t)_{t\geq0}$ be two
independent processes
generated by \eqref{eq:generateur} and starting respectively from
$\mathbf{x}$ and $\mathbf{y}$.
Let us denote
\[
\tau_{\mathrm{in}} = \inf \{ t\geq0 | I_t = J_t=
i_0 \} \quad\mbox{and}\quad\tau_{\mathrm{out}} = \inf \{ t\geq\tau
_{\mathrm{in}} | I_t \neq i_0 \mbox{ or }
J_t \neq i_0 \}.
\]
For every $b,c>0$ such that $b>c$, we define
\[
p_{c,b}(\mathbf{x},\mathbf{y})= \p (\tau_{\mathrm{in}}< c, \tau
_{\mathrm{out}} > b ).
\]
By Assumptions~\ref{hyp:a-reg} and~\ref{hyp:rec}, we have
$p_{c,b}(\mathbf{x},\mathbf{y}) > 0$.
Using the fact that $a$ is bounded, a~coupling argument shows that
$p_{c,b}$ is lower bounded by a positive quantity which only depends on
$i$ and~$j$.
We then obtain the bound
\begin{eqnarray*}
\E \bigl[\widetilde{d}(\mathbf{X}_t,\mathbf{Y}_t) \bigr]
&\leq&\E \bigl[ \one_{ \{ \tau_{\mathrm{in}}
< c, \tau_{\mathrm{out}} > b  \} } \widetilde{d}(\mathbf{X}_t,
\mathbf{Y}_t) \bigr] + 1-p_{c,b}(\mathbf{x},\mathbf{y})
\\
&\leq& 1 - p_{c,b}(\mathbf{x},\mathbf{y}) \bigl( 1-
\delta^{-1} \mathrm{e}^{\varrho c} \mathrm{e}^{- \rho(i_0) t} d(x,y)
\bigr)
\\
&\leq& 1 - p_{c,b}(\mathbf{x},\mathbf{y}) \bigl( 1-
\delta^{-1} \mathrm{e}^{\varrho c} \mathrm{e}^{- \rho(i_0) t} R
\bigr),
\end{eqnarray*}
where $\varrho$ is given by
\[
\varrho= - \min\bigl\{q \alpha(k) | k \in F\bigr\}.
\]
There exist $c>0$ and $t_*>c$ such that $ 1- \delta^{-1} \mathrm
{e}^{\varrho c}
\mathrm{e}^{- \rho(i_0) t_*} R >0$. Since $F$ is finite, we can
furthermore bound
$p_{c,b}$ from below by the minimum over all $i,j\in F$,
and the result follows for any $b>t_*$ and $t^* \in(t_*,b)$.
\end{pf}

%
%re3.13 #&#
\begin{Rq}
One can see from this proof that it is not necessary that the jump
rates are lower bounded, as in Assumption~\ref{hyp:rec}. Indeed, we
need that, for each $i,j \in F$, the jump times of $I$ are
stochastically smaller than a variable which does not depend of the
dynamics of $X$.
\end{Rq}

%s3.5 #&#
\subsection{Proofs of Theorem \texorpdfstring{\protect\ref{th:W-noncst}}{1.5} and Theorem
\texorpdfstring{\protect\ref{th:W-noncst2}}{3.3}}

Recall that Lemmata~\ref{lem:Curv-Lyap} and~\ref{lem:Lyap-noncst}
yield the existence of a Lyapunov\vspace*{2pt} function $V=V_{x_0}$, for some
$x_0\in E$, Lemma~\ref{lem:d-contract} shows that $\widetilde{d}$ is
contracting for $\mathbf{P}$, and Lemma~\ref{lem:d-small} proves that
sublevel sets of $V$ are $\widetilde{d}$-small. So we can use Theorem~\ref{th:Weak-Harris} to deduce that there exist a probability measure
$\boldsymbol{\pi}$ and some constants $C, \lambda, t_0>0$ such that,
for all $t \ge t_0$,
\[
\mathcal{W}_{\widetilde{\mathbf{d}}} (\boldsymbol{\mu} \mathbf{P}_t,
\boldsymbol{\pi} ) \leq C \mathrm{e}^{-\lambda t} \mathcal{W}_{\widetilde{{\mathbf{d}}}} (
\boldsymbol{\mu}, \boldsymbol{\pi} ),
\]
for every probability measure $\boldsymbol{\mu}$ on $\mathbf{E}$. In
this expression, $ \widetilde{\mathbf{d}}$ is defined by
\[
\widetilde{\mathbf{d}}(\mathbf{x},\mathbf{y}) = \sqrt{\bigl(\one
_{i\neq
j} + \one_{i = j} \bigl(1 \wedge d^q(x,y)\bigr)
\bigr) \bigl( 1 + d^q(x,x_0) + d^q(y,x_0)
\bigr)},
\]
where $\mathbf{x}=(x,i)$, $\mathbf{y}=(y,j)$ belong to $\mathbf{E}$,
$x_0$ is as in Assumption~\ref{hypo:Curv} and $q\in(0,1]$.
%We
%conclude the proof by noting
%that $\mathbf{d} \leq\widetilde{\mathbf{d}}$.
Noting that $\mathbf{d} \leq\widetilde{\mathbf{d}}$ we conclude that
for $t \ge t_0$ one has
\[
\mathcal{W}_{\mathbf{d}} (\delta_{\mathbf{y}_0} \mathbf {P}_t,
\boldsymbol{\pi} ) \leq C \mathrm{e}^{-\lambda t} \biggl(1+\sum
_{i\in F} \int_E d(y_0,x)^q
\boldsymbol{\pi} (\mathrm{d}x,i) \biggr).
\]
Since furthermore
\[
\mathcal{W}_{\mathbf{d}} (\delta_{\mathbf{y}_0} \mathbf {P}_t,
\boldsymbol{\pi} ) \leq1,
\]
for all $t\leq t_0$, this ends the proof.

%s4 #&#
\section{Two special cases}

Here, we give some sufficient conditions allowing to verify our main
assumptions in situations where
the underlying processes are deterministic or diffusive. Note that we
can find sufficient conditions in \cite{Clo12} for stochastically
monotone processes, in \cite{CJ12} for birth--death processes and in
\cite{E11} for diffusion processes.

%s4.1 #&#
\subsection{The case of diffusion processes}
\label{sect:dif}

Let us recall that a diffusion process on $\R^d$, $d\in\N^*$, is a
process generated by\vspace*{-1pt}
%
%
%e4.1 #&#
\begin{equation}
\label{eq:gen-diffusion} \forall x \in\R^d,\qquad\mathcal{L} f(x) = \sum
_{i=1}^d b_i (x) \, \partial
_i f(x) + \sum_{i,j=1}^d
\bigl(\sigma(x)\sigma(x)^{t}\bigr)_{i,j}\, \partial
_{i,j} f(x),
\end{equation}
where $f$ is a smooth enough function and $b, \sigma$ are regular
enough, say
%
%
%e4.2 #&#
\begin{equation}
\label{eq:hypo} \forall x,y \in\R^d,\qquad\bigl \Vert\sigma(x) - \sigma(y)
\bigr \Vert+ \bigl \Vert b(x) -b(y) \bigr \Vert\leq K \Vert x-y \Vert
\end{equation}
for some $K>0$. In the previous expression, $\Vert\cdot\Vert$
denotes both the Euclidean norm and the subordinate norm.
%

%le4.1 #&#
\begin{lem}
\label{lem:WC-diff}
Let $(P_t)_{t\geq0}$ be the Markov semigroup generated by \eqref
{eq:gen-diffusion}. If $\sigma$ is constant and
%
%
%e4.3 #&#
\begin{equation}
\label{e:assb} \forall x,y\in\R^d,\qquad\bigl\langle b(x) - b(y), x-y
\bigr\rangle\leq- \alpha \Vert x -y \Vert^2,
\end{equation}
for some $\alpha\in\R$, then
\[
\forall t \geq0, \qquad\mathcal{W}_{\Vert\cdot\Vert} (\mu P_t, \nu
P_t) \leq\mathrm{e}^{-\alpha t} \mathcal{W}_{\Vert\cdot\Vert} (\mu,
\nu),
\]
for any probability measures $\mu$ and $\nu$.
\end{lem}

\begin{pf}
This is an immediate consequence of \eref{e:assb}. One can see this by
using the same Brownian motion for two different
solutions of the SDE starting with different initial measures.
\end{pf}

If $\sigma$ is not constant, then one can also use \cite{BCGMZ}, Proposition~6.1,
which essentially requires the Lipschitz constant of $\sigma$ to be
sufficiently small
compared to the rate of contraction $\alpha$, see also \cite
{RS05,E11}. The ``small level sets'' assumption of Theorem~\ref
{th:dtv-cst} or Theorem~\ref{th:dtv-noncst} is satisfied if one of the
underlying diffusions verifies H\"ormander's hypoellipticity assumption
and satisfies furthermore a natural controllability assumption. See,
for instance, \cite{H11} for an introduction on this subject.

%
%re4.2 #&#
\begin{Rq}[(Exponential convergence for an infinite dimensional process)]
The previous result gives also the convergence for switching
Fokker--Planck processes. Indeed, we can consider that each underlying
Markov process $(Z^{(i)}_t)_{t\geq0}$ is deterministic, belongs to the
space of smooth density functions, and verifies
\[
\partial_t Z^{(i)}_t (x) = \sum
_{k=1}^d - \partial_k
\bigl(b_k Z^{(i)}_t\bigr) (x) + \sum
_{k,l=1}^d \partial_{k,l} \bigl(
\sigma_{k,l} Z^{(i)}_t\bigr) (x)
\]
for all $x\in\R^d$, and $t\geq0$. The previous lemma gives a
contraction as in Assumption~\ref{hypo:Curv}, for each underlying
process, where $d$ is the Wasserstein metric.
\end{Rq}

%s4.2 #&#
\subsection{Case of piecewise deterministic Markov processes}
\label{sect:PDMP}

Let us assume that each one of the underlying Markov processes is
actually deterministic. More precisely, we consider that $E$ is an open
of $\R^d$, $d\in\N^*$ and $\mathcal{L}^{(i)} f = G^{(i)} \cdot
\nabla f $, for every $i\in F$, where $(G^{(i)})_{i\in F}$ is a family
of vector fields such that the ordinary differential equations
$x'=G^{(i)}(x)$ have a unique and global solution for any initial
condition, for every $i\in F$. Lemma~\ref{lem:WC-diff} gives the
assumption in order to apply Theorem~\ref{th:W-cst} and Theorem~\ref
{th:W-noncst}. In general, we cannot apply Theorem~\ref{th:dtv-cst}
or Theorem~\ref{th:dtv-noncst2} but \cite{BH12,MLMZ-Horm} give a
sufficient condition ensuring that $\mathbf{X}$ generates densities:

%
%as4.3 #&#
\begin{hypo}[(H\"{o}rmander-type bracket conditions)]
\label{hyp:Hormander}
Let $\mathcal{G}_0=\{ G^{(i)} -G^{(j)}, i\neq j \}$ and for all $k\geq0$,
\[
\mathcal{G}_{k+1}= \bigl\{ \bigl[G^{(i)},G\bigr] | i\in F, G
\in\mathcal{G}_k \bigr\},
\]
where $[\cdot,\cdot]$ designs the Lie bracket. We have $\mathcal
{G}_k(x) = \{
G(x) | G \in\mathcal{G}_k \}= \R^d$, for every $x\in E$.
\end{hypo}

In this case our main result gives
the following theorem.
%

%th4.4 #&#
\begin{theo}
Let us suppose that Assumptions~\ref{hyp:a-reg}, \ref{hyp:rec} and
\ref{hyp:Hormander} hold. If one of the two following assumptions is satisfied:
\begin{itemize}
\item$a(x,i,j)$ does not depend to $x$ and $I$ is ergodic with an
invariant measure $\nu$ satisfying
\[
\sum_{i\in F} \nu(i) \lambda(i)>0;
\]
\item Assumption~\ref{hyp-iborne} holds and
\[
\sum_{i\in F} \nu(i) \alpha(i)>0,
\]
for some increasing sequence $\alpha$ satisfying $\alpha(n) \leq\min_{i \in F_n} \lambda(i)$, for all $n\leq\bar{n}$
\end{itemize}
then there exist a probability measure $\boldsymbol{\pi} $ and three
constants $C, \lambda, t_0 >0$ such that
\[
\forall t \geq t_0,\qquad d_\TV (\delta_{\mathbf{x}}
\mathbf{P}_t, \boldsymbol{\pi} ) \leq C \mathrm{e}^{-\lambda t}
\bigl(1+V(x)\bigr),
\]
for every $\mathbf{x}=(x,i) \in\mathbf{E}$.
\end{theo}

\begin{pf}
Using \cite{MLMZ-Horm}, Theorem~6.6, we see that compact sets are
small for $\mathbf{X}$. Using Lemma~\ref{lem:Lyap-cst} in the first
case and Lemma~\ref{lem:Lyap-noncst} in the second case, we see that
we can apply Theorem~\ref{th:Harris}.
\end{pf}

%s5 #&#
\section{Examples}
\label{sect:exemple}

Here, we give three simple examples to illustrate our results.

%s5.1 #&#
\subsection{The most elementary example}

Let us consider the example where $X$ belongs to $\R$ and verifies
\[
\forall t\geq0, \qquad\partial_t X_t = I_t
X_t,
\]
where $(I_t)_{t\geq0}$ is the continuous time Markov chain, on $\{
-1,1\}$, which jumps from $1$ to $-1$ with rate $a_1>0$ and from $-1$
to $1$ with rate $a_{-1}>0$.
If $a_1>a_{-1}$ then Theorems~\ref{th:W-cst} and~\ref{th:W-noncst}
give the exponential
ergodicity of $\mathbf{X}$ in the Wasserstein distance. Here, the
invariant law is
\[
\delta_0 \otimes\frac{1}{a_{-1}+a_1} (a_{1}
\delta_{-1} + a_{-1} \delta_1),
\]
and there is clearly no convergence in total variation. Thus, classical
Harris' theorem does not work here. Furthermore, the classical law of
large number gives
\[
\lim_{t \rightarrow+ \infty} X_t = \cases{ %
0 &\quad$
\mbox{a.s., if } a_1>a_{-1}$,
\cr
+ \infty&\quad$\mbox{a.s.,
if } a_1<a_{-1}$. }
\]

In particular, there is no convergence when $a_1<a_{-1}$.

%
%re5.1 #&#
\begin{Rq}
In our main theorems, we use a Wasserstein distance associated to a
distance comparable to $d^q$ rather than $d$. We choose this distance
because, in general, moments of $\mathbf{X}$ can explode even though
$\mathbf{X}$ converges in law. For instance, in the above
example, one has $\lim_{t \to\infty} \E X_t = \infty$ as soon as
$a_{1}<1$. See also \cite{BGM} for
comments on the optimal choice of the parameter $q$.
\end{Rq}

%s5.2 #&#
\subsection{Wasserstein contraction of some switching dynamical systems}

Let us consider a slight generalisation of the previous example; that
is $X$ belongs to $\R$ and verifies
%
%
%e5.1 #&#
\begin{equation}
\label{eq:positive-mean} \forall t\geq0,\qquad\partial_t X_t = -
a(I_t) X_t,
\end{equation}
where $(I_t)_{t\geq0}$ is a recurrent continuous time Markov chain on
a finite state space $F$ and $a$ a function from $F$ to $\R$. Theorem~\ref{th:W-cst} gives the exponential--Wasserstein ergodicity under the
condition that
%
%
%e5.2 #&#
\begin{equation}
\label{eq:positive-mean2} \sum_{i \in F} a(i) \nu(i)>0,
\end{equation}
where $\nu$ is a invariant measure of $I$. This simple example
satisfies a bound like in Assumption~\ref{hypo:Curv}. Indeed we have
the following lemma.
%

%le5.2 #&#
\begin{lem}
If \eqref{eq:positive-mean} and \eqref{eq:positive-mean2} are
satisfied then there is a distance $\boldsymbol{\delta}$ on $\mathbf
{E}$ such that the Wasserstein curvature of the semigroup of $\mathbf
{X}$ is positive, that is, there exists $\lambda>0$ such that
\[
\forall t\geq0,\qquad\mathcal{W}_{\boldsymbol{\delta}} (\delta _{\mathbf{x}}
\mathbf{P}_t, \delta_{\mathbf{y}} \mathbf{P}_t ) \leq
\mathrm{e}^{- \lambda t} \boldsymbol{\delta}(\mathbf {x},\mathbf{y}),
\]
for all $\mathbf{x},\mathbf{y} \in\mathbf{E}$.
\end{lem}

\begin{pf}
First, let us give a complement on the conclusion of Lemma~\ref
{lem:Per-Frob}. The Markov chain $I$ satisfies its assumptions and
using the results of \cite{BGM}, there exist a function $\psi$ on
$F$, $\rho>0$ and $p\in(0,1)$ verifying
\[
\forall t \geq0,\qquad\mathbb{E} \bigl[ \psi(I_t)
\mathrm{e}^{-
\int_0^t p
a(I_s) \,\mathrm{d}s} \bigr] = \mathrm{e}^{- \rho t} \mathbb{E} \bigl[
\psi(I_0) \bigr].
\]
Now let $\boldsymbol{\delta}$ be the distance, on $\mathbf{E}$,
defined by
\[
\forall\mathbf{x},\mathbf{y} \in\mathbf{E}, \qquad\boldsymbol {\delta} (
\mathbf{x}, \mathbf{y}) = \one_{\{i=j\}} \psi(i) |x-y|^p +
\one_{\{i \neq j\}} \frac
{\overline
{\psi}}{\underline{\psi}} \bigl(\psi(i) |x|^p + \psi(j)
|y|^p+1\bigr),
\]
where
\[
\overline{\psi}= \max_{k \in F} \psi(k)\quad\mbox{and}\quad
\underline {\psi}= \min_{k \in F} \psi(k).
\]
Now, using the fact that for all $t>0$ one has
\[
X_t= X_0 \mathrm{e}^{- \int_0^t a(I_s) \,\mathrm{d}s},
\]
the proof is straightforward.
\end{pf}

%s5.3 #&#
\subsection{Surprising blow-up under exponential ergodicity assumptions}

Here we give some comments on \cite{BBMZ-Plan}, Example~1.4, which
also illustrate the sharpness of our criteria. Let us consider $E=\R
^2$, $F=\{0,1\}$, $\mathcal{L}^{(i)} f = A_{i} \cdot\nabla f$ where
\[
A_0 = %
\pmatrix{ -1 & 3
\cr
-1/3 & -1 } %
\quad\mbox{and}\quad A_1 = %
\pmatrix{ -1 & -1/3
\cr
3 & -1
} %
,
\]
$ a(x,0,0)=a(x,1,1)=0$, and $a(x,1,0)=a(x,0,1)=a>0$, for all $x\in\R
^2$. In short, $\mathbf{X}$ is generated, for all $x\in\R^2$ and $i
\in\{0,1\}$, by
%
%
%e5.3 #&#
\begin{equation}
\label{eq:gen-spiral} \mathbf{L} f(x,i) = A_i \cdot\nabla f(x,i) + a
\bigl( f(x,1-i) - f(x,i) \bigr).
\end{equation}
Since $a$ does not depend on its first component, $I$ is a Markov
process and it converges exponentially to
\[
\nu= \tfrac{1}{2} \delta_0 + \tfrac{1}{2}
\delta_1.
\]
For each $i\in\{0,1\}$, we have $\partial_t Z^{(i)}_t = A_i
Z^{(i)}_t$ and thus we easily prove that
%
%
%e5.4 #&#
\begin{equation}
\label{eq:curv-spirale} \bigl\llVert Z^{(i)}_t \bigr\rrVert
_i \leq\mathrm{e}^{-t} \bigl\llVert Z^{(i)}_0
\bigr\rrVert _i \quad\mbox{and}\quad\bigl\llVert
Z^{(i)}_t \bigr\rrVert _{1-i} \leq3
\mathrm{e}^{-t} \bigl\llVert Z^{(i)}_0 \bigr
\rrVert _{1-i},
\end{equation}
for every $t\geq0$, where the norms $\Vert\cdot\Vert_0$ and $\Vert
\cdot\Vert_1$ are defined by
\[
\forall u = (u_1,u_2) \in\R^2,\qquad\llVert
u \rrVert _0= \sqrt {(u_1/3)^2 +
u_2^2} \quad\mbox{and}\quad\llVert u \rrVert
_1= \sqrt {u_1^2 +
(u_2/3)^2}.
\]
Thus each flow $i \in\{0,1\}$ contracts, with the norm $\Vert\cdot
\Vert_i$, and converges geometrically, with the norm $\Vert\cdot
\Vert_{1-i}$, to the same limit. Nevertheless, if $a$ is large enough
then \cite{BBMZ-Plan}, Example~1.4, shows that
\[
\lim_{t\rightarrow+ \infty} \Vert X_t \Vert= + \infty.
\]
In particular, the conclusion of Theorem~\ref{th:W-cst} is not
satisfied. This
illustrates the fact that assuming that each underlying dynamics
converges geometrically is not
sufficient in general to guarantee the convergence of $X$. Moreover,
this shows that it is essential in
Theorem~\ref{th:W-cst} to measure the constants $\rho(i)$ with
respect to the \textit{same}
distance for every $i$. Note that the Wasserstein curvature of
$Z^{(i)}$, with respect to $\Vert\cdot\Vert_{1-i}$, is negative and
given by $-37/3$.

%s5.4 #&#
\subsection{Non-convergence when $I$ is recurrent but not positive recurrent}

A last example is the following: the process $X$ verifies
\[
\forall t \geq0,\qquad\mathrm{d} X_t = - (X_t-
a_{I_t}) \,\mathrm{d}t,
\]
where $(a_n)_{n\geq0}$ is a bounded real sequence and $I$ is an
irreducible and recurrent continuous time Markov chain which is not
positive recurrent. It is easy to see that the sequence of laws of
$(X_t)_{t\geq0}$ is tight and we can hope that there exists a
probability measure $\pi$ verifying
\[
\lim_{t\rightarrow+ \infty} \E \bigl[ f(X_t) \bigr] = \int f \,
\mathrm{d}\pi,
\]
for every continuous and bounded function $f$ and any starting
distribution. But in general, this is false. To illustrate it, let us
consider the case when $I$ is the classical continuous-time random walk
on $\N$ reflected at $0$. Namely, $I$ is generated by
\[
J f(i) = \tfrac{1}{2}f(i+1) + \tfrac{1}{2}f(i-1) - f(i),
\]
if $i\neq0$ and
\[
J f(0) = f(1) - f(0).
\]
The sequence $a$ on the other hand is defined recursively by:
\[
a_{n+1} = \cases{ %
a_n&\quad$\mbox{if }n \notin
\bigl\{2^k | k \in\N\bigr\}$,\vspace*{2pt}
\cr
- a_n&\quad$\mbox{if }n
\in\bigl\{2^k | k \in\N\bigr\}$. }
\]
In this case, the central limit theorem gives that $I_t \approx\sqrt {t}$ and so, for very large
times, $I$ and $a$ do not switch on the same time scale. As a matter of fact,
the process $a_{I_t}$ stays constant during longer and longer stretches
of time.
It is then possible to find two sequences of \textit{deterministic}
times $(t_n)_{n\geq0}$ and $(s_n)_{n\geq0}$, both converging to
infinity, and such that
\[
\lim_{n \rightarrow+ \infty} \E \bigl[ f(X_{t_n}) \bigr] = f(0) \quad
\mbox{and}\quad\lim_{n \rightarrow+ \infty} \E \bigl[ f(X_{s_n}) \bigr]
= f(1).
\]
Thus this process exhibits ageing and is not exponentially stable, even
though there
exists $C>0$, such that for any two starting points $\mathbf{x}=(x,i)$
and $\mathbf{y}=(y,j)$, we have
\[
\forall t\geq0,\qquad\mathcal{W}_{\mathbf{d}_0} (\delta_\mathbf{x}
\mathbf{P}_t,\delta_\mathbf{y} \mathbf{P}_t) \leq
\frac{C}{\sqrt {t}} |i-j|,
\]
where $\mathbf{d}_0 ( \mathbf{x}, \mathbf{y}) = \one_{i=j} \Vert x
- y
\Vert\wedge1 + \one_{i\neq j}$.

\section*{Acknowledgements}

We would like to thank the referees for their careful reading of the manuscript.
Support for MH's research was provided by the Royal Society
through a Wolfson Research Merit Award. BC was partially supported by
Ecole Doctorale MSTIC, Universit\'e Paris-Est, and by ANR MANEGE (09-BLAN-0215).

%suskaldyti doi

% imsref loaded by audrone.aklyte, 2014-03-03 16:16:11
%
% imsref loaded by audrone.aklyte, 2014-03-04 10:33:14

\printhistory

\end{document}